\documentclass[11pt]{amsart}
\usepackage[utf8]{inputenc}
\usepackage{bm}
\usepackage{fontenc}
\usepackage{amsfonts}
\usepackage{amssymb}
\usepackage{amsmath}
\usepackage{amsthm}\usepackage{mathtools}
\usepackage{enumerate}
\usepackage{enumitem} 
\usepackage[pagebackref,colorlinks,linkcolor=blue,citecolor=blue,urlcolor=blue,hypertexnames=true]{hyperref}
\usepackage[pagebackref,hypertexnames=true]{hyperref}
\usepackage{enumitem}
\usepackage{mathrsfs}
\usepackage{tikz}
\usetikzlibrary{calc}
\usepackage{marginnote}
\usepackage{xcolor,enumitem}
\usepackage{soul}\usepackage{tikz}
\usepackage[all,cmtip]{xy} \usepackage{caption}

\numberwithin{equation}{section}
\newtheorem{theorem}{Theorem}
\newtheorem{thm}[theorem]{Theorem}

\newtheorem{cor}[theorem]{Corollary}
\newtheorem{lemma}[theorem]{Lemma}
\newtheorem{prop}[theorem]{Proposition}
\theoremstyle{definition}
\newtheorem{definition}[theorem]{Definition}

\newtheorem{problem}[theorem]{Problem}
\theoremstyle{remark}
\newtheorem{remark}[theorem]{Remark}

\numberwithin{theorem}{section}

\newcommand {\Z}{\mathbb Z}

\newcommand{\eps}{\varepsilon}

\newcommand{\cF}{\mathcal{F}}

\newcommand{\cP}{\mathcal{P}}

\newcommand{\SB}{\mathrm{SB}}
\newcommand{\SD}{\mathrm{SD}}
\newcommand{\SF}{\mathrm{SF}}

\newcommand{\Q}{\mathbb{Q}}
\newcommand{\N}{\mathbb{N}}
\newcommand{\R}{\mathbb{R}}

\begin{document}

\title[Descriptive set theory of separable Fréchet spaces]{Descriptive set theory of separable Fréchet spaces}


\author{ B. M. Braga}
\address{IMPA, Estrada Dona Castorina 110, 22460-320, Rio de Janeiro, RJ, Brazil}\email{demendoncabraga@gmail.com}

\author{ W. Corrêa}
\address{Departamento de Matemática, Instituto de Ciências  Matemáticas e de  Computação, Universidade de São Paulo,
Avenida Trabalhador São Carlense 400, 13566-590, São Carlos, SP, Brazil.}\email{willhans@icmc.usp.br}

\author{ V. Ferenczi}
\address{Departamento de Matemática, Instituto de Matemática e Estatística, Universidade de São Paulo, rua do Matão 1010, 05508-090 São Paulo SP, 
Brazil, and Equipe d’Analyse Fonctionnelle, Institut de Mathématiques de Jussieu, Sorbonne Université,  4 place Jussieu - Boîte courrier 247, 75252 Paris Cedex 05, 
France.}\email{ferenczi@ime.usp.br}

\thanks{B. M. Braga  was partially supported by FAPERJ, grant E-26/200.167/2023,  by CNPq, grant 303571/2022-5, and by Serrapilheira, grant R-2501-51476. W. Corrêa was supported by FAPESP, grants 2023/06973-2 and 2023/12916-1, and by CNPq, grant 304990/2023-0. V. Ferenczi  was supported by FAPESP, grant 2023/12916-1, and by CNPq, grant 304194/2023-9.}
 
\date{}

\begin{abstract}
In the  past few decades, much has been done regarding the descriptive set theory of  separable Banach spaces. However, the descriptive properties of separable Fréchet spaces have not yet been investigated. In these notes, we look at this problem, its relation with the (now standard) theory for separable Banach spaces, and we compute/estimate the descriptive complexity of some classical classes of separable Fréchet spaces such as Fréchet-Hilbert, Schwartz, nuclear, and Montel spaces. Our main result shows that the class of Montel spaces is complete coanalytic. Noticeably, this applies outside the realm of descriptive set theory and solves an old problem regarding   Fréchet spaces satisfying the Heine--Borel property (i.e., Montel spaces). Precisely, we show that there is no separable Montel space containing isomorphic copies of all  separable Montel spaces. 
\end{abstract}

 \maketitle
\begin{center}
    \emph{Dedicated to the memory of Joe Diestel}
\end{center}

\section{Introduction}\label{intro1}

The relation of isomorphism between locally convex topological vector spaces is notoriously complicated even if one restricts to the more tamed class of Banach spaces, which makes any sort of classification program for this class of spaces virtually impossible (see \cite{Bossard2002,FerencziLouveauRosendal2009JLMS}). However, as it has already been done for the class of separable Banach spaces, descriptive set theory provides an appropriate framework to formalize and study classes  or properties of spaces (see \cite{Bossard1993French,Bossard2002,DodosBook2010}). While Banach spaces form an important class of locally convex topological spaces, there are many naturally occurring spaces in functional analysis and partial differential equations which are not Banach; e.g.,   the space of entire functions on the complex plane, the space of infinitely differentiable functions on $[0,1]$, the Schwartz space of functions with rapidly decreasing derivatives, and  the space of continuous functions on $\sigma$-compact Hausdorff spaces.  

The main goal of this article is to extend the study of descriptive complexity  to the larger class of Fréchet spaces. With this goal in mind, a fair amount of attention is given to carefully developing the theory and part of it is at times an adaptation of known techniques to our scenario. We  highlight here   our most interesting result, which is far from a mere adaptation and requires novel insights: we show that the class of \emph{Montel spaces} --- i.e.,   Fréchet spaces which satisfy the Heine--Borel property ---  is complete coanalytic (Theorem \ref{ThmMonCompleteCoanalytic}). Computing the exact complexity of a non-Borel property is usually very challenging and this is indeed the case for Montel spaces (we return to this with more details later in this introduction).  

The computation that the class of Montel spaces is complete coanalytic has an important consequence which goes beyond the realm of descriptive set theory, we explain this now. A class of separable Fréchet spaces is said to have an \emph{universal element} if there is a space in this class which contains isomorphic copies of all other spaces in this class. The first universality result in Fréchet theory is perhaps the Mazur--Orlicz theorem stating that the space of continuous functions on $\R$, $C(\R)$, is universal for all separable Fréchet space (see \cite{MazurOrlicz1948Studia}   or \cite[Page 101]{RolewiczSecondEdition1984}). Later, motivated by Grothendieck's seminal   work on nuclear spaces (see \cite{Grothendieck1955MAMS}), T. Komura and Y. Komura obtained their celebrated result showing that the class of nuclear spaces has a universal element (see  \cite{KomuraKomura1965MathAnn}\footnote{See   \cite[Corollary 29.9]{MeiseVogt1997Book} for a version in English.}    and Section \ref{someclasses} for more details).   From the beginning of the 70's onward, universality in Fréchet spaces turned into a topic of intensive research. For instance, on the opposite direction than for nuclear spaces, V. Moscatelli showed that the class of (Fréchet) Schwartz spaces has no universal element (see \cite{Moscatelli1975ComptesRendus} or \cite[Proposition on page 277]{Bellenot1980Compositio}). Moreover,   universality for Fréchet spaces is  one of the main topics of the classic monograph \cite{RolewiczSecondEdition1984} by S. Rolewicz   (see      also \cite{Rosenberger1973TAMS,Kalton1977Studia,Moscatelli1978Crelle,Jarchow1981Book,MeiseVogt1997Book} for more on universality of classes of Fréchet spaces). However, despite the interest in this area, classic methods have not been able to deal   with Montel spaces; which, together with nuclear and Schwartz spaces, is arguably one of the most important   classes of Fréchet spaces (which are not Banach). As an application of the techniques   developed herein, we show that the class of separable Montel spaces contains no universal elements (Corollary \ref{CorNoMontelUniversal}).  This should emphasize   the strength of descriptive methods in analysis.\\

 As one of our focuses is to relate the,  now classic, descriptive set theory of separable Banach spaces with the theory for Fréchet spaces developed herein, we start by quickly recalling some of its basics (for a more careful overview and precise definitions, we refer the reader to  Section \ref{intro}). The class of all separable Banach spaces is not a set and, therefore,  one starts by coding this class in an appropriate fashion: the space of continuous functions on the Cantor set, $C(\Delta)$, is isometrically universal for the class of separable Banach spaces and hence 
 \[\SB=\{X\subseteq C(\Delta)\mid X\text{ is a closed vector subspace}\},\] endowed with the  Effros Borel structure, becomes a standard Borel space which serves as a coding for the class of all separable  Banach spaces. In fact, although there are  other natural codings for this class of spaces (e.g., as quotients of $\ell_1$, as sequences in $C(\Delta)$, etc), all ``reasonable'' codings  give rise to the same theory --- see \cite[Proposition 2.8]{Bossard2002} or Section  \ref{appendix} for details (see also \cite{CuthDolezalDouchaKurka2022,CuthDolezalDouchaKurka2024}).
These tools allow us to study descriptive properties of classes of Banach spaces. For instance, we have:
\begin{enumerate}

\item [1.]  $\langle \ell_p \rangle = \{X\in\SB\mid X\cong \ell_p\}$ is Borel for every $1 < p < \infty$ (see \cite[Page 130]{Bossard2002} and \cite{Godefroy2023}),
\item[2.] $\langle c_0\rangle=\{X\in\SB\mid X\cong c_0\}$ is complete  analytic  (see \cite[Theorem 1.1]{Kurka2019isomorphism}),
\item[3.]  $\rm{REFL}=\{X\in\SB\mid X\text{ is reflexive}\}$ is complete coanalytic   (see \cite[Theorem 2.5]{DodosBook2010}), and
\item[4.]  $\SD=\{X\in\SB\mid X^*\text{ is separable}\}$ is complete  coanalytic  (see \cite[Theorem 2.11]{DodosBook2010}).
\end{enumerate} 
See also \cite{Godefroy2017} for more examples of spaces with a Borel isomorphism class. Besides the intrinsic interest in understanding the precise complexity of Banach space properties, these methods have important applications outside the strict point of view of descriptive set theory; for example, these methods have deep applications to universality problems (see \cite{Bossard2002,ArgyrosDodos2007AdvMath,Dodos2009TAMS,DodosFerenczi2007FundMath,Dodos2010Studia,Braga2014CzMathJ}).
 
As Fréchet spaces are not \emph{metric} but rather \emph{metrizable}, it does not make sense to deal with metric properties such as the existence of an \emph{isometrically universal} space for all separable Fréchet spaces. However, isomorphic properties behave well in this category and there are  \emph{isomorphically universal} separable Fréchet spaces, such as $C(\R)$. Our study of the complexity of Fréchet space properties starts then by coding all separable Fréchet spaces as 
\[\SF=\{X\subseteq C(\R)\mid X\text{ is a closed vector subspace}\},\]
endowed with the Effros Borel structure (see Section \ref{SFcoding}). Conceptually, the first question to address is how $\SB$ and $\SF$ relate to each other. Precisely, it is of upmost importance (1) to  compute the complexity of the spaces in $\SF$ which are Banach spaces in disguise and (2) to understand whether $\SB$ and $\SF$ are compatible when  coding classes of Banach spaces. We deal with both these issues in Section \ref{SFcoding}.  For (1), we show that 
\[\SB_{\rm F}=\{X\in \SF\mid X\ \text{ is normable}\}\]
is Borel (see Proposition \ref{SBborelSF}). For (2),  we show that if $\cP$ is an isomorphic Banach space property, then its coding $\rm P$ and $\rm P_{\rm F}$ as subsets in $\SB$ and $\SF$, respectively, i.e., 
\[\rm P=\{X\in \SB\mid X\ \text{ satisfies }\ \mathcal P\} \ \text{ and }\ \rm P_{\rm F}=\{X\in \SF\mid X\ \text{ satisfies } \ \mathcal  P\},\]
coincide (see Theorem \ref{consistent}). This shows, in particular, that the study of $\SF$ does indeed generalize the one of $\SB$.

In view of the previous paragraph, the remainder of the article is dedicated to the study of the space of Fréchet spaces which are not normable, i.e., are not isomorphic to a Banach space. We write
\[\SF_{\neg \rm B}=\SF\setminus \SB_{\rm F}\]
to denote this space. As $\SB_{\rm F}$ is Borel (see Proposition \ref{SBborelSF}), $\SF_{\neg \rm B}$ is also a standard Borel space. Hence, in order to understand $\SF$, it is indeed enough to study $\SB$ and  $\SF_{\neg \rm B}$ separately. We start our study of $\SF_{\neg \rm B}$ by presenting a method that allows us to conclude that the ``non Banach counterpart'' of certain well-studied Banach space properties such as reflexivity and having a separable dual are also non Borel (see Corollary \ref{CorollarySDandReflNotBorel}). These results are later strengthened  considerably and we show that both these classes are $\Pi^1_1$-hard, meaning that every coanalytic subset Borel reduces to it (see Theorems \ref{ThmREFLnotBanachPI11Hard} and \ref{ThmSDnotBanachPI11Hard}).

From Subsection \ref{SubsectionFH} onward, we investigate some classical properties of Fréchet spaces such as Fréchet-Hilbert, Schwartz, nuclear and Montel spaces. Apart from the former, all the remaining properties are genuinely non Banach in the sense that only finite dimensional Banach spaces can satisfy them. As for Fréchet-Hilbert spaces, besides finite dimensional spaces, spaces isomorphic to $\ell_2$ also satisfy this property; however, $\R^\N$, $\ell_2^\N$, and $\ell_2\oplus \R^\N$ are also Fréchet-Hilbert and not normable. We show that the following classes of separable Fréchet spaces have Borel codings in $\SF$:
 \begin{enumerate}
 \item [I.] $\rm FH=\{X\in \SF\mid X\ \text{ if Fréchet-Hilbert}\}$ (Theorem \ref{ThmFHBorel}),
 \item [II.] The isomorphism class of $\R^\N$, $\ell^\N_2$, and $\ell_2\oplus \R^\N$ (Theorem \ref{ThmFHBorel}),
 \item[III.] $\mathrm{Sch}=\{X\in \SF\mid X\text{ is a Schwartz space}\}$ (Theorem \ref{schwartz}), and 
 \item[IV.] $\mathrm{Nuc}=\{X\in \SF\mid X\text{ is a nuclear  space}\}$ (Theorem \ref{nuclear}).
 \end{enumerate}
We also show that some classes are non Borel and quite complicated in nature. For instance, in Theorem \ref{ThmComplexityCR}, we show that the isomorphism class of $C(\R)$ is complete analytic and, in particular, non Borel. 

Our main results come from studying the complexity of the class of Montel spaces. Recall that a Fréchet space is called \emph{Montel} if it satisfies the Heine--Borel property, i.e., if every bounded closed subset of it is compact. Philosophically speaking, the Heine--Borel property  can be seen as a counterpart of being Banach: while on the one hand, Banach spaces generalize to  infinite dimensional spaces  the geometry of finite dimensional spaces given by norms, the Heine--Borel property generalizes their topological behavior to the infinite dimensional scenario. In Subsection 
\ref{SubsectionMontel}, we show that
\[\mathrm{Mon}=\{X\in \SF\mid X\ \text{ is Montel}\}\]
is a complete coanalytic subset and, in particular, non Borel (see Theorem \ref{ThmMonCompleteCoanalytic}). As a consequence, 
we show that there is no  isomorphically universal space for the class of separable Montel spaces  (Corollary \ref{CorNoMontelUniversal}). Moreover, we obtain an even stronger result: if a separable Fréchet space $X$ contains isomorphic copies of all separable Montel spaces, then it must also contain isomorphic copies of $c_0$ and of all $\ell_p$'s, for $p\in [1,\infty)$ (see Corollary \ref{CorNoTrulyFrechetUniversal.V3}).

In order to obtain the results in the previous paragraph about Montel spaces, we use K\"{o}the matrices (see Defintions \ref{DefinitionKothe} and \ref{DefinitionKotheMontel}) and a characterization of when such matrices give rise to Montel spaces by Dieudonné and Gomes (see Theorem \ref{ThmDieudonneGomes}). Moreover, we also need the classical result that if $X$ is a Polish space which is not $\sigma$-compact, than the space of all sequences in $X$ without convergent subsequences is complete coanalytic (see Proposition \ref{PropMKComCoanalitic} for further details). With these two tools in hand and choosing $X$ appropriately, we are able to construct a Borel reduction of a complete coanalytic set into the space of Montel spaces (see Lemma \ref{LemmaFunctionReducingMontel}). It should be noticed that this Borel reduction is one of the main novelties of this article. In the usual descriptive set theory of Banach spaces, such Borel reductions are   obtained with the help of trees of natural numbers and using that the well-founded trees are a complete coanalytic space; so, the ill-founded trees are complete analytic and the existence of branches will usually imply that a determined Banach space will ``appear'' in this reduction. This sort of reduction does not seem to work for Fréchet spaces which are not Banach so a novel style of reduction is needed.

We finish this paper with a section about open problems (Section \ref{SecOpenProb}) and a short appendix explaining how different codings for separable Fréchet spaces will usually not interfere in their descriptive theory (Section \ref{appendix}).

\section{Notation and Background}\label{intro}

\subsection{Notation} 
We write $\N=\{1,2,...\}$, $\R^+=\{x\in\R\mid x>0\}$, and $\Q^+=\{x\in\Q\mid x>0\}$.  Every time an index is omitted the reader should interpret the index as being the set of natural numbers $\N$. So, we will   often write  $(x_j)_j$ and  $\sum_j x_j$ instead of $(x_j)_{j\in\N}$  and $\sum_{j\in\N} x_j$, respectively. 

\subsection{Descriptive set theory} Here we recall the basics of descriptive set theory; for an excellent monograph on the subject, we refer the reader to \cite{KechrisBook1995}.
A separable topological space $X$ is said to be a \emph{Polish space} if there exists a complete metric on $X$ generating its topology. A continuous image of a Polish space into another Polish 
space is called an \emph{analytic set} and a subset of a Polish space whose complement is analytic is called \emph{coanalytic}. A measure space $(X,\mathcal{A})$, where $X$ is a set and $\mathcal{A}$ is a 
$\sigma$-algebra of subsets of $X$, is called a \emph{standard Borel space} if there exists a Polish topology on this set whose Borel $\sigma$-algebra coincides with $\mathcal{A}$. We 
define Borel, analytic and coanalytic sets in standard Borel spaces by saying that these are the sets that, by considering a compatible Polish topology, are Borel, analytic, and 
coanalytic, respectively. Observe that this is well defined, i.e., this definition does not depend on the Polish topology itself but only on its Borel structure. A function between two 
standard Borel spaces is called \emph{Borel measurable} if the inverse image of each Borel subset of its codomain is Borel in its domain. We usually refer to Borel measurable functions 
just as Borel functions.

Given a Polish space $X$, $\cF(X)$ denotes the set of all of its nonempty closed subsets. We endow $\cF(X)$  with the \emph{Effros Borel structure}, i.e., the $\sigma$-algebra 
generated by
\[\{F\subseteq X \mid F\cap U\neq \emptyset\},\]
where $U$ varies among the open sets of $X$. It can be shown that $\cF(X)$ with the Effros Borel structure is a standard Borel space (\cite[Theorem 12.6]{KechrisBook1995}).  The following well-known lemma (see 
\cite[Theorem 12.13]{KechrisBook1995}) will be crucial in some of our proofs.

\begin{lemma}[Kuratowski-Ryll-Nardzewski] 
Let $X$ be a Polish space. There exists a sequence of Borel functions $(S_n)_{n\in\N}\colon\cF(X)\to X$ such that $\{S_n(F)\}_{n\in\N}$ is dense in $F$, for all closed $F\subseteq X$. Those functions are called Borel selectors.\label{LemmaKuratowski-Ryll-Nardzewski}
\end{lemma}

\subsection{The standard coding for separable Banach spaces}\label{SubsectionSB}
In this subsection, we give a brief introduction to the standard coding for separable Banach spaces (see
\cite{Bossard2002} for a detailed treatment of it).

Letting  $\Delta$ denote the Cantor set, the Banach space $C(\Delta)$ of continuous real-valued functions on $\Delta$ is \emph{isometrically universal} for the class of all separable Banach spaces, i.e., every separable Banach space is isometrically isomorphic to a closed linear subspace of $C(\Delta)$ (see \cite[Page 79]{KechrisBook1995}). Therefore, the class of all separable Banach spaces can be coded as the set of closed subspaces of $C(\Delta)$ and we define
\[\text{SB}=\{X\subseteq C(\Delta)\mid  X \text{ is a closed vector subspace}\}.\]
As $C(\Delta)$ is   a Polish space, 
$\cF(C(\Delta))$  is a stardard Borel space and it can be shown that $\SB$ is a Borel set in $\cF(C(\Delta))$ and hence it is also a standard Borel space (see \cite[Theorem 2.2]{DodosBook2010}\footnote{Alternatively, see Section \ref{SFcoding}. We present a proof of this argument for the space of separable Fréchet spaces which is an analog of the argument for $\SB$.}). In fact, there is nothing special about $C(\Delta)$ here. Precisely, if $X$ is any separable Banach space, then 
\[\SB(X)=\left\{Y\subseteq X\mid Y\text{ is a closed vector subspace of }X\right\}\]
is a standard Borel space equipped with its Effros Borel structure. With this notation, $\SB=\SB(C(\Delta))$.

This coding for the class of separable Banach spaces allows us to  wonder if specific classes of separable Banach spaces are Borel, analytic, coanalytic, etc. We refer to   \cite{DodosBook2010} and \cite{Bossard2002}, for detailed expositions on the coding $\SB$, such as the computation of the complexity of many classes of separable Banach spaces as sets in this coding.  

\subsection{Separable Fréchet spaces}\label{SubsectionSepFrechet}

A completely metrizable locally convex topological vector space $X$ is called a \emph{Fréchet space} if its topology can be defined by an invariant metric. If $X$ is a Fréchet space, then there exists a sequence $(\|\cdot\|_n)_{n\in\N}$ of pseudonorms  generating  the topology of $X$, i.e., a sequence $(x_j)_j$ in $X$ converges to zero if and only if   $\lim_j\|x_j\|_n=0$ for all $n\in\N$.  Similarly, a sequence $(x_j)_j$ is Cauchy if and only if it is Cauchy with respect to $\|\cdot\|_n$ for all $n\in\N$. Moreover, we can always assume the sequence $(\|\cdot\|_n)_n$ to satisfy \[\|\cdot\|_n\leq\|\cdot\|_{n+1},\ \text{ for all }\ n\in\N.\]

Given a Fréchet space $X$ and an increasing sequence of pseudonorms $(\|\cdot\|_n)_n$ generating the topology of $X$, we define a distance $d\colon X^2\to[0,\infty)$ as
\[d(x,y)=\sum_{n\in\N}2^{-n}\frac{\|x-y\|_n}{1+\|x-y\|_n},\]
 for all $x,y\in X$. Then $d$ is a complete invariant metric generating the topology of $X$. 
 Define 
 \begin{align*}\|\cdot\|_F\colon X&\to [0,\infty)\\ 
 x&\mapsto d(x,0)
 \end{align*}
  We call $\|\cdot\|_F$ an \emph{F-norm} for $X$.\footnote{Notice that, despite the terminology and notation, $\|\cdot\|_F$ is not an actual norm.}

We finish this section presenting a few basic results in Fréchet theory which will be  crucial in what follows. Although the next lemma is completely elementary, we present its proof for lack of a precise citing source. 

\begin{lemma}\label{continuous}
Let $X$ and $Y$ be Fréchet spaces, and let $(\|\cdot\|_{X,n})_{n\in\N}$ and $(\|\cdot\|_{Y,n})_{n\in\N}$ be   increasing sequences of pseudonorms generating the  
 topologies of $X$ and $Y$, respectively. An operator $T\colon X\to Y$ is continuous if and only if for all $n\in\N$, there are $K>0$ and $m\in\N$, such that 
\[\|T(x)\|_{Y,n}\leq K\|x\|_{X,m},\]
 for all $x\in X$.
\end{lemma}
 
\begin{proof}
For the backwards implication, it is evident that if a sequence $(x_m)_m$ in $X$ goes to zero with respect to all pseudonorms  $(\|\cdot\|_{X,n})_{n\in\N}$, then the equation in the lemma implies that its image under $T$ also goes to zero in $Y$.  For the forward implication, suppose $T$ is continuous and that the desired conclusion fails. Then, we can pick $n\in\N$ and a sequence $(x_m)_m$ in $X$ such that 
\begin{equation}\label{Eq20125.1}\|T(x_m)\|_{Y,n}\geq m\|x_m\|_{X,m}
\end{equation}
for all $m\in\N$. As $(\|\cdot\|_{X,m})_m$ is an increasing sequence,  $(x_m/(m\|x_m\|_{X,m}))_m$ converges to zero with respect to all these pseudonorms, so, it converges to zero in $X$. As $T$ is continuous, its image under $T$ should also converge to zero in $Y$, however, it does not by \eqref{Eq20125.1}.\end{proof} 

Given two Banach spaces $X$ and $Y$, we say that two sequences $(x_j)_j\in X^\N$ and $(y_j)_j\in X^\N$   are \emph{equivalent}, and write $(x_j)_j\sim(y_j)_j$, if there is $K>0$ such that  
\[\frac{1}{K}\left\|\sum_{i=1}^\ell a_iy_i\right\|\leq\left\|\sum_{i=1}^\ell a_ix_i\right\|\leq K\left\|\sum_{i=1}^\ell a_iy_i\right\|,\]
 for all $\ell\in\N$, and all $a_1,...,a_\ell\in\R$. Similarly, given   Fréchet spaces $X$ and $Y$ whose topologies are generated by   increasing sequences of pseudonorms $(\|\cdot\|_{X,n})_{n\in\N}$ and $(\|\cdot\|_{Y,n})_{n\in\N}$, we say that two sequences $(x_j)_j\in X^\N$ and $(y_j)_j\in X^\N$  are \emph{equivalent}, and  write $(x_j)_j\sim(y_j)_j$, if for all $n\in\N$, there are $K>0$ and $m\in\N$, such that
\[\left\|\sum_{i=1}^\ell a_ix_i\right\|_{X,n}\leq K\left\|\sum_{i=1}^\ell a_iy_i\right\|_{Y,m},\ \text{ and }\ \left\|\sum_{i=1}^\ell a_iy_i\right\|_{Y,n}\leq K\left\|\sum_{i=1}^\ell a_ix_i\right\|_{X,m},\]
 for all $\ell\in\N$, and all $a_1,...,a_\ell\in\R$. We write $X \cong Y$ to mean that the Fréchet spaces $X$ and $Y$ are linearly isomorphic. The proof of the next lemma is elementary and we omit it.

\begin{lemma}\label{basic}
Let $X$ and $Y$ be  separable Fréchet spaces, and ${(\|\cdot\|_{X,n})_{n\in\N}}$ and $(\|\cdot\|_{Y,n})_{n\in\N}$ be  increasing sequences of pseudonorms generating the topologies of $X$ and $Y$, respectively. Then $X\cong Y$ if and only if there exist $(x_j)_j\in X^\N$ and $(y_j)_j\in Y^\N$, such that $(x_j)_j\sim(y_j)_j$,  $\overline{\text{span}}\{x_j\mid j\in\N\}=X$, and $\overline{\text{span}}\{y_j\mid j\in\N\}=Y$.
\end{lemma}

Let $X$ be a Fréchet space. A subset $K\subseteq X$ is called \emph{bounded} if for all neighborhoods of zero $U$ there exists $a\in \R$ such that $K\subseteq aU$, where \[aU=\{au\in X\mid u\in U\}.\] The following is a well-known characterization of when Fréchet spaces are mere Banach spaces in disguise (see \cite[Theorem 3.2.2]{RolewiczSecondEdition1984}).

\begin{thm}\label{ThmFréchetIsBanachIFFBoundedN}
A  Fréchet space $X$ is a Banach space, i.e., there is a norm $\|\cdot\|$ generating its topology, if and only if $X$ has a bounded neighborhood of zero.
\end{thm}

\section{The standard Borel space of separable Fréchet spaces}\label{SFcoding}

Let us denote by $C(\R)$ the space of continuous real valued functions on $\R$ endowed with the topology of uniform convergence on compact subsets of $\R$. So,   $C(\R)$  is a separable Fréchet space and the sequence of pseudonorms $(\|\cdot\|_n)_n$ given by 
\begin{equation}\label{Eq.Standard.PseudoNormCR}\|f\|_n=\underset{x\in[-n,n]}{\sup}|f(x)|,\end{equation}
for all $n\in\N$ and all $f\in C(\R)$, 
 generates the topology of $C(\R)$. The sequence $(\|\cdot\|_n)_n$ above will be called the \emph{standard sequence of pseudonorms on $C(\R)$}. If we let 
\[\|f\|_{F}=\sum_{n\in\N}2^{-n}\frac{\|f\|_n}{1+\|f\|_n},\]
 $\|\cdot\|_F$ is an $F$-norm for $C(\R)$. From now on, $\|\cdot\|_F$ will always refer to this $F$-norm on $C(\R)$.


It is well known (see \cite{MazurOrlicz1948Studia} or \cite[page 101]{RolewiczSecondEdition1984}) that  $C(\R)$ is \emph{isomorphically  universal} for all separable Fréchet spaces, i.e., every separable Fréchet space can be isomorphically embedded in $C(\R)$. Therefore, a natural coding for the class of separable Fréchet spaces is 
\[\SF=\{X\subseteq C(\R)\mid X\text{ is a closed vector subspace}\}.\]
 The natural Borel structure to endow $\SF$ with is the Effros Borel structure defined earlier, i.e., the smallest $\sigma$-algebra containing 
\[\{X\in \SF\mid X\cap U\neq\emptyset\},\]
 where $U$ varies among all the open sets of $C(\R)$. The coding $\SF$ with the Effros Borel structure is a standard Borel space. Indeed, the same proof that $\SB$ is a standard Borel space works for $\SF$, i.e., let  $(U_n)_n$ be a countable basis for the topology of $C(\R)$, then
\begin{align*}
X  \in  \SF &  \ \\
\Leftrightarrow \  
&\forall n \big(0\in U_n\Rightarrow X\cap U_n\neq\emptyset\big)\ \text{ and }\ \forall \ell,m,k\in\N,\forall r,t\in\Q\\ 
&\big(rU_\ell+tU_m\subseteq U_k\ \&\ X\cap U_\ell\neq\emptyset\ \&\ X\cap U_m\neq\emptyset\big)\Rightarrow \big(X\cap U_k\neq\emptyset\big).
\end{align*}

\begin{remark}
    Of course this is not the only way of coding the class of separable Fréchet spaces. In the appendix of these notes we discuss other ways of coding separable Fréchet spaces and, following \cite[Proposition 2.8]{Bossard2002}, we explain why any ``reasonable enough coding''  will provide us with the same theory.  
\end{remark}

As every Banach space is a Fréchet space, we can see the separable Banach spaces as a subset of $\SF$. We say $X\in\SF$ is \emph{normable} if there exists a norm $\|\cdot\|$ on $X$ generating its topology, i.e., such that $\mathrm{Id}\colon (X,\|\cdot\|)\to (X,\|\cdot\|_F)$ is an isomorphism. We define $\SB_{\rm F}$ as 
\[\SB_{\rm F}=\{X\in \SF\mid X\text{ is normable}\}.\]

\begin{remark}
     It is important to notice that, unlike the Banach spaces case, it does not make sense to talk about isometrically universal Fréchet spaces or isometric properties of Fréchet spaces. Indeed, Fréchet spaces are locally convex topological vector spaces which admit a complete and invariant metric. But they are not metric spaces per se. 
\end{remark}

\subsection{Relation between the complexity of sets in $\SB$ and $\SF$}\label{SubsectionSBandSF.OK}

As every Banach space is a Fréchet space, this gives us two ways of coding a given (isomorphically invariant) class of separable Banach spaces, (i) as a subset of $\SB$ and (ii) as a subset of $\SF$. It is natural to expect that the complexity of such classes would be independent of this choice. In this subsection, we show that this is indeed the case (Theorem \ref{consistent}). Therefore, this new theory  is consistent with the already known theory for separable Banach spaces.

Let $(\|\cdot\|)_n$ be the standard sequence of pseudonorms generating the topology of $C(\R)$ (see \eqref{Eq.Standard.PseudoNormCR}). For each  $n\in\N$, $r\in \Q^+$, and $X\in \SF$, we let
\[B_{n,r}=\{f\in C(\R)\mid\|f\|_n < r\}\ \text{ and }\  B_{n,r}(X)=B_{n,r}\cap X.\] 
It is easy to see that $\{B_{n,r}(X)\}_{n,r}$ is a basis for the topology of $X$. We also let 
\[\overline{B}_{n,r}=\overline{B_{n,r}}\ \text{ and }\ \overline{B}_{n,r}(X)=\overline{B_{n,r}(X)}
.\]

\begin{lemma}\label{bolas}
Let $n\in\N$ and $r\in \R^+$. The map \[X\in \SF\mapsto \overline{B}_{n,r}(X) \in \cF(C(\R))\] is Borel.
\end{lemma}

\begin{proof}
Call this function $\varphi$. We only need to notice that, for an open set $U\subseteq C(\R)$, we have
\[\varphi^{-1}\big(\{F\in \cF(C(\R))\mid F\cap U\neq\emptyset\}\big)=\{X\in\SF\mid  X\cap U\cap B_{n,r} \neq\emptyset\}.\]
 As the right-hand side of the equality above is Borel, the lemma follows. 
\end{proof}

\begin{prop}\label{SBborelSF}
$\SB_{\rm F}$ is Borel in $\SF$.
\end{prop}

\begin{proof}
By Theorem \ref{ThmFréchetIsBanachIFFBoundedN}, a  Fréchet space is normable (i.e., Banach) if and only if there exists a bounded neighborhood of zero. Therefore,
\begin{align*}
X\in\SB_{\rm F}\ &\Leftrightarrow \ \exists q\in\Q^+,\ \exists n\in\N,\ 
\ \forall m\in\N,\ \exists p\in\Q^+\\
&\ \ \ \ \ \overline{B}_{n,q}(X)\subseteq \overline{B}_{m,p}(X).
\end{align*} 
In order to conclude that $\SB_{\rm F}$ is Borel notice that, for all standard Borel spaces $M$, the set $\{(F_1,F_2)\in\cF(M)^2\mid F_1\subseteq F_2\}$ is Borel (see \cite[Exercise 12.11.ii]{KechrisBook1995}). Since, 
by Lemma \ref{bolas}, the assignment 
\[X\in \SF\mapsto (\overline{B}_{n,q}(X),\overline{B}_{m,p}(X))\in \cF(C(\R))^2\] is Borel,  we conclude that    $\SB_{\rm F}$ is Borel.
\end{proof}

By a \emph{class $\cP$ of separable Banach spaces} (resp., \emph{separable Fréchet spaces}) we shall always mean a class which is closed under isomorphisms, i.e., if $X\cong Y$, and $X\in \cP$, then $Y\in\cP$. The following  shows that the coding $\SF$ for separable Fréchet spaces is ``consistent'' with the coding for separable Banach spaces $\SB$.

\begin{thm}\label{consistent}
There exists a Borel injection $\Psi\colon \SB\to \SF$ such that \[\Psi(X)\cong X\ \text{ for all }\ X\in\SB.\] In particular, if $\cP$ is a class of separable Banach spaces, and $\rm P$ and $\rm P_{\rm F}$ are its coding in $\SB$ and $\SF$, respectively, i.e., \[\rm P=\{X\in\SB\mid X\in\cP\}\ \text{  and }\  \rm P_{\rm F}=\{X\in \SF\mid X\in\cP\},\] then the complexity of $\rm P_{\rm F}$ in $\SF$ is the same as the  complexity of $\rm P$ in $\SB$.
\end{thm}

We will need a couple of  auxiliary lemmas to prove Theorem \ref{consistent}. Although those lemmas have analogous versions (and proofs) for Banach spaces, we present them here for the reader's convenience (for the Banach space case, see \cite{DodosBook2010}).

\begin{lemma}\label{phi}
The function $\Phi\colon C(\R)^\N\to\SF$ given by letting \[\Phi((f_j)_j)=\overline{\mathrm{span}}\{f_j\mid j\in\N\},\] for all $(f_j)_j\in C(\R)^\N$, is Borel.
\end{lemma}

\begin{proof}
Fix an open set $U\subseteq C(\R)$. For each $(a_1,...,a_n)\in\Q^{<\N}$, define a map $\varphi_{(a_1...,a_n)}\colon C(\R)^\N\to C(\R)$ by letting \[\varphi_{(a_1,...a_n)}((f_j)_j)=\sum_{i=1}^na_if_i,\] for all $(f_j)\in C(\R)^\N$. Then,
\begin{align*}
\Phi^{-1}(\{X\in\SF\mid & X\cap U\neq \emptyset\})\\
&=\{(f_j)_j\in C(\R)^\N\mid \overline{\text{span}}\left\{f_j\mid j\in\N\}\cap U\neq\emptyset\right\}\\
&=\bigcup_{(a_1,..,a_n)\in\Q^{<\N}}\varphi^{-1}_{(a_1,...,a_n)}(U).
\end{align*}
 As each $\varphi_{(a_1,...,a_n)}$ is clearly Borel, we are done.
\end{proof}

\begin{lemma}\label{LemmaEquivSeq}
The set \[\left\{((f_j)_j,(g_j)_j)\in C(\R)^\N\times C(\R)^\N\mid (f_j)_j\sim(g_j)_j\right\}\]
is Borel in $C(\R)^\N\times C(\R)^\N$.
\end{lemma}

\begin{proof}
Since
\begin{align*}
(f_j)_j\sim & (g_j)_j\ \Leftrightarrow\\
&\forall n\in\N,\exists K\in\Q^+,\exists m\in\N,\forall (a_1,...,a_\ell)\in\Q^{<\N}\\
&\left\|\sum_{i=1}^\ell a_if_i\right\|_n\leq K\left\|\sum_{i=1}^\ell a_ig_i\right\|_m\ \&\  \left\|\sum_{i=1}^\ell a_ig_i\right\|_n\leq K\left\|\sum_{i=1}^\ell a_if_i\right\|_m,
\end{align*}
the described set is clearly Borel.
\end{proof}

\begin{lemma}\label{iso}
The set \[\{(X,Y)\in\SF^2\mid X\cong Y\}\] 
is analytic in $\SF^2$. In particular, 
\[\langle X\rangle=\{Y\in \SF\mid Y\cong X\}\]
is analytic for all $X\in \SF$.
\end{lemma}

\begin{proof}
By Lemma \ref{basic},  $X\cong Y$ if and only if there exists $(f_j)_j\in X^\N$ and $(g_j)_j\in Y^\N$ such that $(f_j)_j\sim (g_j)_j$, $\overline{\text{span}}\{f_j\mid j\in\N\}=X$, and $\overline{\text{span}}\{g_j\mid j\in\N\}=Y$. Define a map  
\begin{align*}\Phi\colon  C(\R)^\N\times C(\R)^\N&\to   \SF\\
 ((f_j)_j,(g_j)_j) &\mapsto  (\overline{\text{span}}\{f_j\mid j\in\N\},\overline{\text{span}}\{g_j\mid j\in\N\}) \end{align*}
and notice that, if  \[B=\{((f_j)_j,(g_j)_j)\in C(\R)^\N\times C(\R)^\N\mid (f_j)_j\sim(g_j)_j\},\] then 
\[\Phi(B)=\{(X,Y)\in\SF^2\mid X\cong Y\}.\] Therefore, by Lemmas \ref{phi} and \ref{LemmaEquivSeq}, and  as Borel images of Borel sets are analytic, we are done.
\end{proof}

\begin{proof}[Proof of Theorem \ref{consistent}]
It follows from the classical theory of Banach spaces that there is a linear isomorphism  $\varphi\colon C(\Delta)\to C[0,1]$ (see \cite[Theorem 4.4.8]{AlbiacKaltonBook})  This isomorphism then induces a Borel isomorphism 
\[\Phi\colon \SB\to \SB(C[0,1])\]
 such that $ \Phi(X)\cong X$, for all $X\in\SB$. Let 
\[C_F[0,1]=\{f\in C(\R)\mid f(t)=f(0), \forall t<0, f(t)=f(1), \forall t>1\},\]
 so, $C_F[0,1]\in\SF$. Notice also that $\SB(C_F[0,1])=\{X\in\SF\mid X\subseteq C_F[0,1]\}$ is Borel in $\SF$. Indeed, this follows since  $\{F\in\cF(M)\mid F\subseteq F_0\}$ is Borel, for all Polish spaces $M$ and all $F_0\in\cF(M)$ (see \cite[Exercise 12.11.ii]{KechrisBook1995}). Let now   $\lambda\colon C[0,1]\to C_F[0,1]$ be the isomorphism given by 
 letting
 \[\lambda(f)(t)=\left\{\begin{array}{ll}
 f(t)   ,  &\text{ if } t\in [0,1]  \\
f(0),      &\text{ if }t<0\\
f(1), &\text{ if } t>1.
 \end{array}\right.\]
Clearly, this induces a Borel isomorphism 
\[\Lambda\colon \SB(C[0,1])\to \SB(C_F[0,1]).\]
Then, letting  \[\Psi=\Lambda\circ \Phi,\] we obtain a   map $\SB\to \SF$ which is a Borel isomorphism onto $ \SB(C_F[0,1])$  and such that  $X\cong \Psi(X)$ for all $X\in\SB$. 

Say $\cP$ is a class of separable Banach spaces and let \[\rm P=\{X\in\SB\mid X\in\cP\}\ \text{ and }\ \rm P_{\rm F}=\{X\in\SF\mid X\in \cP\}.\] We need to show that $\rm P$ is $\Sigma^1_n$ (resp.\ $\Pi^1_n$) if and only if $\rm P_{\rm F}$ is $\Sigma^1_n$ (resp.\ $\Pi^1_n$). Suppose $\rm P$ is $\Sigma^1_n$. Then $\Psi(\rm P)$ is $\Sigma^1_n$ in $\SF$, and, by Lemma \ref{iso}, \[\langle\Psi(\rm P)\rangle=\{X\in\SF\mid \exists Y\in \Psi(\rm P),Y\cong X\}\] is $\Sigma^1_n$. As we clearly have that $\rm P_{\rm F}=\langle\Psi(\rm P)\rangle$, we conclude that $P_F$ is $\Sigma^1_n$.

Suppose now that $\rm P_{\rm F}$ is $\Sigma^1_n$. Then, as $\SB(C_F[0,1])$, the set \[A=\rm P_{\rm F}\cap  \SB(C_F[0,1])\]  is $\Sigma^1_n$. As we clearly have that $\rm P=\Psi^{-1}(A)$, we conclude that $\rm P$ must also be $\Sigma^1_n$.

The result for $\Pi^1_n$ is analogous, we just need to work with the appropriate complements. Hence, we also have that $\rm P$ is $\Delta^1_n$ if and only if $\rm P_{\rm F}$ is $\Delta^1_n$. 
\end{proof}

\begin{remark}
    As an immediate consequence of Theorem \ref{consistent}, the result of \cite[Theorem 5]{FerencziLouveauRosendal2009JLMS} regarding $\SB$ extends to the class of separable Fréchet spaces: i.e., the relation of isomorphism on $\SF$ has the highest possible complexity among analytic equivalence relations on Polish spaces (see \cite{FerencziLouveauRosendal2009JLMS} and the references therein for the theory of classification of such equivalence relations by Borel reducibility).
\end{remark}

\section{The space of non-normable  separable   Fréchet spaces}
 In view of Theorem \ref{consistent}, the descriptive set theory of $\SB$ and $\SB_{\mathrm F}$ coincide. Hence, in order to better develop the descriptive set theory of $\SF$, we are left to study 
 \[\SF_{\neg\rm B}=\SF\setminus \SB_{\mathrm F},\]
 the space of all separable Fréchet space which are not Banach, i.e., not normable. This will be the focus of this section.

\subsection{Reflexive Fréchet spaces and Fréchet spaces with separable dual}
 We start looking at two of the most important properties of Banach spaces in the more general context of Fréchet spaces: being reflexive and having a separable dual. Both these properties are well known to be complete coanalytic in $\SB$  (see \cite[Theorems 2.5 and 2.11]{DodosBook2010}). In this subsection, we shall study the complexity of these properties in $\SF_{\neg\rm B}$ by a simple method, which can easily be adapted to show that other properties are also not Borel. We will revisit these classes in Subsection \ref{SubsectionSDREFL} below and obtain better results by using more subtle arguments.

  We start fixing a  concrete embedding of $C(\R)^{\N}$ inside $C(\R)$. First,   recall that if $(X_n)_n$ is a sequence  of Fréchet spaces, then, 
  their product $\prod_{n\in\N}X_n$ is also a Fréchet space; $\prod_{n\in\N}X_n$ is considered endowed with the usual product topology. In case all $X_n$ equal a single Fréchet space, say $X_n=X$ for all $n\in\N$, we simply write $X^\N$ for  $\prod_{n\in\N}X_n$.
  
  We now describe our embedding of  $C(\R)^{\N}$ in  $C(\R)$. Given a prime $p$ and $n\in\N$, we set  
  \[I_{p, n} = \left[p^n + \frac{1}{4}, p^n + \frac{3}{4}\right],\  I_{p, n}^- = \left[p^n, p^n + \frac{1}{4}\right], \text{ and }\ I_{p, n}^+ = \left[p^n + \frac{3}{4}, p^n + 1\right],\] and let $\varphi_{p, n} \colon I_{p, n} \to [-n, n]$ be a continuous bijection. Then, for a prime $p$, we let $i_{p} \colon C(\R) \rightarrow C(\R)$ be the embedding defined by
\[
i_p(f)(x) = \begin{cases}
    f(\varphi_{p, n}(x)), & \mbox{if } x \in I_{p, n} \\
    4f(-n)(x - p^n), & \mbox{if } x \in I_{p, n}^- \\
    -4f(n)(x - p^n - 1), & \mbox{if } x \in I_{p, n}^+ \\
    0, & \mbox{otherwise.}
\end{cases}
\]
Finally, letting  $(p_k)_k$ be an enumeration of the prime numbers, we  define  an embedding $i \colon C(\R)^{\N} \rightarrow C(\R)$ by
 letting 
 \begin{equation}\label{Eq.Map.i.Emb}
i((f_k)_k) = \sum_{k\in\N} i_{p_k}(f_k)
\end{equation}
for all $(f_k)_k\in C(\R)^\N$. Notice that $i$ is well defined. Indeed,  if $k ,l\in\N$  are distinct, then the supports of $i_{p_k}(f_k)$ and of $i_{p_l}(f_l)$ are disjoint.

\begin{lemma}\label{LemmaMapXtoXNBorel} Let $\Psi$ and $i$ be the maps  from Theorem \ref{consistent} and \eqref{Eq.Map.i.Emb}, respectively. Then, the map $\Pi\colon \SB \rightarrow \SF$ defined
by letting \[\Pi(X)=i (\Psi(X)^\N),\]
for all $X\in \SB$, is Borel.
\end{lemma}
\begin{proof}
Since $\Psi$ is Borel, it is sufficient to show that the map \[j \colon X \in \SF\mapsto i(X^\N)\in \SF\] is Borel. But this is straightforward since, for a given open subset $U \subseteq C(\R)$, we have that
\begin{align*}
   j(X) \cap U \neq \emptyset \ \Leftrightarrow \ \exists n_1,\ldots, n_k\in \N\ \text{ such that } \  \sum_{j=1}^ki_{p_j}(S_{n_j}(X))\in U,
\end{align*}
where $(S_n\colon \SF\to C(\R))_n$ is the sequence of Borel selectors given by Lemma \ref{LemmaKuratowski-Ryll-Nardzewski}.
\end{proof}

\begin{cor}\label{CorollarySDandReflNotBorel} Given a Banach space $X$,  $X$  is reflexive (resp., has a separable dual) if and only if $X^\N$ is reflexive (resp., has a separable dual). In particular, $\mathrm{REFL}_{\neg \rm B}$  and $\mathrm{SD}_{\neg \rm B}$ are not Borel.\end{cor}

\begin{proof} The  backwards implication follows since both the properties of being reflexive and having separable dual pass to subspaces and $X$ isomorphically embeds into $X^\N$.  For the forward implication, suppose first that $X$ is reflexive. Then the fact that $X^\N$ is also reflexive is given by \cite[Proposition 25.15]{MeiseVogt1997Book}. On the other hand, if $X$ has separable dual, the same holds for $X^\N$ since, by   \cite[Proposition 24.3]{MeiseVogt1997Book}, its dual is isomorphic to the direct sum
\[\bigoplus_{n\in\N} X^*=\left\{(x^*_n)_n\in (X^*)^\N\mid \exists n_0\in \N,\forall n\geq n_0,\ x_n=0\right\}\]
endowed with the locally convex topology induced by all pseudonorms of the form 
\[\|\cdot\|\colon (x_n^*)_n\in \bigoplus_{n\in\N} X^*\mapsto \sum_{n\in\N}\|x_n^*\|_n,\]
where $(\|\cdot\|_n)_n$ ranges over all sequences such that each $\|\cdot\|_n$ is a continuous pseudonorm on $X^*$. 
As $\bigoplus_{n\in\N} X^*$ is clearly separable if $X^*$ is separable, we conclude that $X^\N$ has separable dual.

The last assumption follows from Lemma \ref{LemmaMapXtoXNBorel} and the fact that the coding for both the reflexive Banach spaces and the Banach spaces with separable dual in $\SB$ are not Borel  (see \cite[Theorems 2.5 and 2.11]{DodosBook2010}).\end{proof}

\subsection{The complexity of Fréchet-Hilbert spaces}\label{SubsectionFH}
A Fréchet space $X$ is called \emph{Hilbertian} if it is isomorphic to a Hilbert space. In particular, Hilbertian spaces are normable. 
A Fréchet space $X$ is called \emph{Fréchet-Hilbert} is there is a sequence $(\|\cdot\|_n)_n$ of pseudonorms on $X$ generating its topology and such that each quotient
\[X/\{x\in X\mid \|x\|_n=0\}\]
is Hilbertian.   We let 
\[\mathrm{FH}=\{X\in \SF\mid X\ \text{ is a Fréchet-Hilbert space}\}.\]
Infinite dimensional separable Fréchet-Hilbert spaces are completely characterized as being isomorphic to either $\ell_2$,  $\R^\N$,  $\ell_2^\N$, or  $\ell_2 \oplus \R^\N$ (see \cite[Theorem 6]{Zarnadze1993Izvestiya}). So, letting $\mathrm{FH}_\infty$ denote the infinite dimensional Fréchet-Hilbert spaces, we have 
\[\mathrm{FH}_\infty=\langle\ell_2\rangle\sqcup\langle \R^\N\rangle\sqcup \langle\ell_2^\N\rangle\sqcup\langle\ell_2\oplus \R^\N\rangle.\]
The isomorphism class of $\ell_2$ has already been studied by Bossard and it follows from Kwapie\'{n}'s theorem that it is Borel (see \cite[Page 130]{Bossard2002}). In this subsection, we will focus on the classes of the remaining Fréchet-Hilbert spaces:
\[\mathrm{FH}_{\neg \rm B}=\langle \R^\N\rangle\sqcup \langle\ell_2^\N\rangle\sqcup\langle\ell_2\oplus \R^\N\rangle.\]
 By the characterization of Fréchet-Hilbert spaces just mentioned, a Fréchet space $X$ belongs to $\mathrm{FH}_{\neg \rm B}$ if and only if it is isomorphic to a direct product $\prod_{n \in \N} H_n$ where each $H_n$ is a Hilbert space with dimension either $1$ or $\aleph_0$.

\begin{lemma}\label{Lemmal2NCharacterization}
 Let $Y$ be a Fréchet space and $(\|\cdot\|_n)_n$ be an increasing sequence   of pseudonorms generating its topology.  For each $m\in\N$,  let  
\[C_m=\{y\in Y\mid \|y\|_m=0\}.\]
Then, given a subspace $X$ of $Y$, we have that
\begin{equation}\label{Eq.l2N.Description.0}X \in \mathrm{FH}_{\neg \rm B}  \Leftrightarrow\ \forall m\in\N\ \Big(X\cap C_m\neq \{0\}\Big)\wedge \Big(X/(X\cap C_m) {\rm\ is\ Hilbertian}\Big).\end{equation}
\end{lemma}

\begin{proof} 
For the forward implication, suppose $X\in \mathrm{FH}_{\neg \rm B} $. So
 $X$ is isomorphic to some direct sum   $H=\prod_{n \in \N}H_n$, where each $H_n$ is a Hilbert space of dimension either $1$ or $\aleph_0$. Let $T\colon H \to X$ be an isomorphism and for each $n\in\N$, let 
\[| \lambda|_n=\max_{i\in \{1,\ldots, n\}}\|\lambda_i\|_{H_i}\]
for all $\lambda=(\lambda_i)_i\in H$. So, $(|\cdot|_n)_n$ is an increasing sequence of pseudonorms of $H$ generating its topology. Fix an arbitrary $m\in\N$ and let us show that $X$ satisfies the right-hand side of \eqref{Eq.l2N.Description.0} for this $m$. As $T$ is continuous,    there are $K>0$ and $n\in \N$ such that 
\[\|T( \lambda)\|_m\leq K|\lambda|_n\]
for all $ \lambda\in H$. In particular, letting 
\[D_n=\{\lambda\in H \mid |\lambda|_n=0\},\]
we have that 
\begin{equation}\label{Eq23Oct24.2.222}T(D_n)\subseteq X\cap C_m\end{equation}
So, $X\cap C_m\neq \{0\}$. Moreover, as $T$ is a surjective map, \eqref{Eq23Oct24.2.222} implies that $T$  induces a well defined surjective map 
\[\lambda+D_n\in H/D_n\mapsto T(\lambda)+X\cap C_m\in X/(X\cap C_m). \]
So, $X/(X\cap C_m)$ is isomorphic to a quotient of $H/D_n $. As the latter  is isomorphic to $ \oplus_{i=1}
^nH_i$,   the open mapping theorem for Fréchet spaces shows that $X/(X\cap C_m)$ is isomorphic to a Hilbert space as desired.

Suppose  $X$ is a subspace of $Y$ satisfying the conditions on the right-hand side of \eqref{Eq.l2N.Description.0}. 
The metric $d$ on $Y$ given by 
\[d(x,y)=\sum_{n\in\N}\frac{1}{2^n}\frac{\|x-y\|_n}{\|x-y\|_n+1},\] 
for all $x,y\in Y$, is a complete metric compatible with the topology of $Y$. As $X\cap C_m\neq \{0\}$ for all $m\in\N$, we can pick a sequence $(x_m)_m$ in $X\setminus \{0\}$ such that $x_m\in X\cap C_m$ for all $m\in\N$. Going to a subsequence if necessary, we can assume without  loss of generality that $(x_m)_m$ is linearly independent. Then, since the  metric $d$ above restricts to a complete metric on $X$, it is immediate that the map 
\[(\lambda_m)_m\in \R^\N\mapsto\sum_{m\in\N}\lambda_mx_m\in X\]
is a well-defined isomorphic embedding. In particular, $X$ contains an isomorphic copy of $\R^\N$ and it cannot be in $\SB_{\rm F}$. 

So we are left to notice that $X\in \rm FH$. But this is immediate since the quotient $X/(X\cap C_m)$ is assumed to be Hilbertian for all $m\in\N$.   \end{proof}

\begin{cor}\label{detail} Let $Y$ be a Fréchet space and $(C_m)_m$ be as in
Lemma \ref{Lemmal2NCharacterization}. If $Y\in \mathrm{FH}_{\neg \rm B} $, then
\[X\in \langle\R^\N\rangle\ \Leftrightarrow \ \forall m\in\N,\ \dim(X/(X \cap C_{m}))<\infty\]
and 
\[X\in \langle\ell_2^\N\rangle\ \Leftrightarrow\ \forall n \in \N, \exists m \geq n,\  \dim( (X \cap C_n)/(X \cap C_m)) =\infty.\]
\end{cor}

\begin{proof} This follows elementarily from Lemma \ref{Lemmal2NCharacterization} and we leave the details to the reader.
\end{proof}

 The next corollary is certainly  well known to experts, but we include it here for the non-expert reader   since it is a nice quick application of the methods above.

\begin{cor}\label{CorSubspaceofRN}
    All infinite dimensional subspaces of $\R^\N$ are isomorphic to $\R^\N$.  
\end{cor}

\begin{proof}
   Suppose $X $ is an infinite dimensional subspace of $\R^\N$. Let $(|\cdot|_n)_n$ be an increasing sequence of pseudonorms on $\R^\N$ generating its topology.  Letting $C_m$ be as in Lemma \ref{Lemmal2NCharacterization} for $Y=\R^\N$, we have that   $\dim(\R^\N/ C_m)<\infty$ for all $m\in\N$ (cf. Corollary \ref{detail}). In particular, \[\dim( X/( X \cap C_m))=\dim(X/  C_m)<\infty\] for all $m\in\N$.
Also, if $X\cap C_m=\{0\}$ for some $m\in\N$, then 
\[x\in X\mapsto x+ X\cap C_m\in X/(X\cap C_m)\]
is injective. But then, as $X$ has infinite dimension, so has $X/(X\cap C_m)$; contradiction. Hence, $X\cap C_m\neq\{0\}$ for all $m\in\N$. By Corollary \ref{detail}, $X$ is isomorphic to $\R^\N$.
\end{proof}

\begin{theorem}\label{ThmFHBorel}
The coding for the class of Fréchet-Hilbert spaces $\rm{FH}$ is Borel as well as the isomorphism classes of $\R^\N$, $\ell_2^\N$, and $\ell_2 \oplus \R^\N$.
\end{theorem}

\begin{proof}
As $\mathrm{FH}$ is the union of the Hilbertian Banach spaces and $\mathrm{FH}_{\neg \rm B} $, it is enough to show that each of these sets are Borel. 
The coding for the Hilbertian Banach spaces in $\SB$ is well-known to be Borel (see \cite[Page 130]{Bossard2002});   this is a trivial consequence of Kwapie\'{n}'s theorem which says that a Banach space is isomorphic to a Hilbert space if and only if it has both type  and cotype 2. Hence, by Theorem \ref{consistent}, the  coding for the Hilbertian Fréchet spaces in $\SF$ is also Borel. So, we are left to compute the complexity of $\mathrm{FH}_{\neg \rm B} $.

Letting $Y=C(\R)$ in  Lemma \ref{Lemmal2NCharacterization}, we show that $\mathrm{FH}_{\neg \rm B} $ is Borel by showing  that  the conditions on the right-hand side of the equivalence in this lemma describe a  Borel subset of $\SF$. For that, let $(S_n\colon \SF\to C(\R))_n$ denote  Borel selectors given by Lemma \ref{LemmaKuratowski-Ryll-Nardzewski}. Then, letting $(C_m)_m$ be as in Lemma \ref{Lemmal2NCharacterization} for  $Y=C(\R)$,  we have that, for all $m\in\N$,  \begin{align} \label{Eq.26Oct24.1}
  X\cap C_m\neq \{0\}  \ \Leftrightarrow\ \exists n\in\N\ \text{ such that } \ S_n(X\cap C_m)\neq 0 .
\end{align} 
So, the first  condition in the right-hand side of \eqref{Eq.l2N.Description.0} is also Borel.
In order to show that the second condition in the right-hand side of \eqref{Eq.l2N.Description.0} is also Borel, notice that the map 
\[X\in \SF\mapsto X/C_m\in \SB(C(\R)/C_m)\]
is clearly Borel. Since $X/(X\cap C_m)\cong X/C_m$, this gives us a Borel realization for the assignment $X\mapsto X/(X\cap C_m)$. The result then follows since being Hilbertian is   a Borel property (see \cite[Page 130]{Bossard2002}). This concludes the proof that $\rm FH$ is Borel.

The fact that $\langle \R^\N\rangle $ and $\langle \ell_2^\N\rangle$ are Borel follows similarly with the extra help of Corollary \ref{detail}.  Precisely,  notice that   
\begin{align}\label{Eq.26Oct24.2}
   \dim(X/  (X\cap & C_m))<\infty \\
\Leftrightarrow\   &\exists  (n_1,\ldots,n_{k})\in \N^{<\infty},\ 
\forall \ell\in\N,\ \forall\eps\in \Q_+,\ \exists  (\lambda_1,\ldots,\lambda_k)\in \Q^{<\infty}  \notag \\
&\left\|S_\ell(X)- \sum_{i=1}^k\lambda_iS_{n_i}(X) \right\|_m\leq \eps \notag
\end{align}
for all $m\in\N$.
Since the  conditions in the right-hand side of  \eqref{Eq.26Oct24.2} are Borel and all quantifiers are countable, Corollary \ref{detail} together with the recently  proved fact  that $\rm FH\setminus \SB_{\rm F}$ is Borel imply  that $\langle \R^\N\rangle$ is Borel. The fact that $\langle \ell_2^\N\rangle$ is Borel follows analogously and we leave the details to the reader. 

Finally, since 
\[\mathrm{FH}_{\neg \rm B} =\langle \R^\N\rangle\sqcup\langle \ell_2^\N\rangle\sqcup \langle \ell_2\oplus \R^\N\rangle,\]
we conclude that $\langle \ell_2\oplus \R^\N\rangle$ must also be Borel.
\end{proof}

\subsection{The complexity of  $C(\R)$}

The result of this section bears similarity with the result that the isomorphism class of $C([0,1])$ in $\SB$
 is complete analytic.
\begin{theorem}\label{ThmComplexityCR}
    The isomorphism class of  $ C(\R)$ in $\SF$   is complete analytic.
\end{theorem}

We start with a simple well-known result about the embeddability of Banach spaces into Fréchet spaces. For lack of a precise reference and for the readers convenience, we include a proof below.

\begin{prop}
\label{PropBanachSpaceEmbFréchet} 
Let $(X,\|\cdot\|)$ be a Banach space,  $Y$ be a Fréchet space, and   $(\|\cdot\|_n)_n$ be an increasing sequence of pseudonorms on $Y$ generating its topology. If $T\colon X\to Y$ is an isomorphic embedding, then there are $L>0$ and $m\in\N$ such that 
\[\frac{1}{L}\|T(x)\|_m\leq \|x\|\leq L\|T(x)\|_m\]
for all $x\in X$.
\end{prop}

\begin{proof}
    As the  inverse of $T$ is continuous, there are
       $L>0$ and $m\in\N$ such that 
\begin{equation}\label{Eq.BanachSubOfFréchet} \|x\|\leq L\|T(x)\|_m\end{equation}
for all $x\in X$. Suppose towards a contradiction that for all $k\in\N$ there is $x_k\in X\setminus \{0\}$ such that  
\[\|x_k\|\leq \frac{1}{k}\|T(x_k)\|_m.\]
As each $x_k$ is non zero, \eqref{Eq.BanachSubOfFréchet} implies the same for each $T(x_k)$ and, by $\|\cdot\|_m$-normalizing it, we can assume furthermore that $\|T(x_k)\|_m=1$ for all $k\in\N$.  So, $(x_k)_k$ is a sequence in $X$ converging to $0$ and $(T(x_k))_k$ is a sequence in $Y$ which does not converge to zero. As $T$ is continuous, this gives us a contradiction.\end{proof}

\begin{prop}\label{PropCFnotIsoCR}
    Let $F\subseteq \R$ be a countable closed subset and $K\subseteq \R$ be compact. Then $C(\R)$ does not isomorphically embed into $C(K\cup F)$.
\end{prop}

\begin{proof}
    Suppose towards a contradiction that $C(\R)$  isomorphically embeds into $C(K\cup F)$. Hence, since the latter isomorphically embeds into $C(K)\oplus C(F)$, we can fix an isomorphic embedding $T\colon C(\R)\to C(K)\oplus C(F)$. Let $T_1\colon C(\R)\to C(K)$ and $T_2\colon C(\R)\to C(F)$ be the compositions of $T$ with the projections of $C(K)\oplus C(F)$ onto $C(K)$ and $C(F)$, respectively. Notice that there is $m\in\N$ such that $T_2\restriction C_m$ is an isomorphic embedding, where
    \[C_m=\{f\in C(\R)\mid \|f\|_m=0\}.\] 
    Indeed, as $T_1$ is bounded, there is $K>0$ and $m\in\N$ such that 
    \[\|T_1(f)\|\leq K\|f\|_m\]
    for all $f\in C(\R)$. So, $T_1\restriction C_m=0$. Hence,  $T_2\restriction C_m=T\restriction C_m$ and $T_2\restriction C_m$ is an isomorphic embedding. 

    As $C[0,1]$ isomorphically embeds into $C_m$, there must be an isomorphic embedding of $C[0,1]$ into $C(F)$. Let $V\colon C[0,1]\to C(F)$ be such an embedding.  As $C[0,1]$ is a Banach space, it follows from Proposition \ref{PropBanachSpaceEmbFréchet} that there is $n\in\N$ such that and $L>0$ such that 
    \[\frac{1}{L}\|V(f)\|_n\leq \|f\|\leq L\|V(f)\|_n\]
    for all $f\in C[0,1]$. But then $C[0,1]$ is isomorphic to a subspace of $C(F\cap [-n,n])$ and this cannot happen since $F$ is countable  (see \cite[Theorem 2.14]{RosenthalHandBookCK}). 
     \end{proof}

 The next result is a  standard result in Banach space theory whose proof works perfectly in the Fréchet space scenario as well, see \cite[Theorem 2.2.3]{AlbiacKaltonBook} (or \cite{Pelczynski1960Studia} for its first appearance).

\begin{theorem}[The Pełczynski decomposition technique]
Let $X$ and $Y$ be Fréchet spaces isomorphic to complemented subspaces of each other. Suppose $X$ is isomorphic to $X^\N$. Then $X$ and $Y$ are isomorphic to each other.\label{ThmDecomposition}
\end{theorem}

\begin{cor}\label{CorDecomposition}
    Let $F\in \{\R,[0,\infty),(-\infty,0]\}$. Then $C(F)$ is isomorphic to $C[0,1]^\N$.
\end{cor}

\begin{proof}
Suppose $F=\R$, the other cases will follow analogously. Since  $C[0,1]^\N$ is clearly isomorphic to $(C[0,1]^\N)^\N$, we only need to show that $C(\R)$ and  $C[0,1]^\N$ are isomorphic to complemented subspaces of each other (Theorem \ref{ThmDecomposition}). For each $n\in\N$, let $T_n\colon C[0,1]\to C[2n,2n+1]$ be an isomorphism and set \[E=\bigsqcup_{n\in\N} [2n,2n+1].\] Then, identifying each $C[2n,2n+1]$ as a subspace of $C(E)$ in the canonical way, the map \[(f_n)_{n\in\N}\in C[0,1]^\N\mapsto \sum_{n\in\N} T(f_n)\in C(E)\]
is an isomorphism. Let $V\colon   C(E)\to C(\R)$ be the map sending each element $f$ in $C(E)$ to its  extension which is affine on each segment of the form $[2n+1,2n+2]$, $n\in\N$, and such that $f(x) = f(0)$ for $x < 0$. Then $V$ is an isomorphic embedding and letting $\chi_E$ be the characteristic function of $E$, the map \[f\in C(\R)\to V(\chi_{E}f)\in C(\R)\] is a continuous projection of $C(\R)$ onto the image of $V$. This shows $C[0,1]^\N$ is   isomorphic to a complemented subspace of $C(\R)$.

On the other hand, let $V$ be the subspace of $\R^\Z$ of families of real numbers indexed by $\Z$ whose coordinate of index $0$ is $0$. Moreover, for each $n\in\N$, let $h_n\colon [0,1]\to [n,n+1]$ be an increasing homeomorphism. Then, the map 
\[\left((\lambda_n)_{n\in\N},f\right)\in V\times C(\R)\mapsto \left(\lambda_n+f\circ h_n\right)_{n\in\Z}\in C[0,1]^\Z\]
is an isomorphism. So, $C(\R)$ is isomorphic to a complemented subspace of $C[0,1]^\Z$. Since the latter is clearly isomorphic to $C[0,1]^\N$, we are done.
\end{proof}

\begin{prop}\label{PropCFIsomCR}
    Let $F\subseteq \R$ be a closed subset such that $F\setminus [-n,n]$ is uncountable for all $n\in\N$. Then $C(F)$ is isomorphic to $C(\R)$.
\end{prop}

\begin{proof}
   Suppose $F\neq \R$, otherwise the conclusion is trivial. Then, by the hypothesis on $F$, we can write $F=\bigsqcup_{i=1}^\infty F_i$ where each $F_i$, for $i\geq 3$, is a compact uncountable subset of $\R$, $F_1$ is either $\emptyset$ or of the form $[\lambda_+,\infty)$ for some $\lambda_+\in \R$, and $F_2$ is either $\emptyset$ or of the form $( \infty,\lambda_-]$ for some $\lambda_-\in \R$. So, $C(F)$ is isomorphic to the Fréchet-sum $\prod_iC(F_i)$. It follows from classical Banach space theory that each $C(F_i)$, for $i\geq 3$, is isomorphic to $C[0,1]$ (see \cite[Theorem 4.4.8]{AlbiacKaltonBook}). Moreover, by Corollary \ref{CorDecomposition}, each of $C(F_1)$ and $C(F_2)$ is either isomorphic to $\{0\}$ or to $  C[0,1]^\N$. Either way, we conclude that $C(F)$ is isomorphic to $ C[0,1]^\N$. By Corollary \ref{CorDecomposition}, $C(F)$ is isomorphic to $C(\R)$ as desired. 
\end{proof}

The next proposition follows the idea of a theorem of Hurewicz, see \cite[Theorem 27.5]{KechrisBook1995}.

\begin{prop}\label{PropFLargeSimga11Complete}
  Let $E\subseteq \R$ be a closed subset such that  $E\setminus (-n,n)$ is uncountable for all $n\in\N$. The set 
    \[\left\{F\in \cF(E)\mid \forall n\in\N,\ F\setminus (-n,n)\text{ is uncountable}\right\}\]
    is $\Sigma^1_1$-complete.
\end{prop}

\begin{proof}
   Let $A$ denote the set in the proposition. First notice that $A$ is analytic. For that, notice that 
   \[\cP=\{H\in \cF(E)\mid H\text{ if perfect}\}\] 
   is Borel (see \cite[Theorem 27.5]{KechrisBook1995}\footnote{The proof that $\cP$ is Borel is explicitly done inside the proof of this theorem.}). Then, since $F\in \cF(E)$ is uncoutable if and only if there is $H\subseteq F$ perfect and non empty (see \cite[Theorem 6.4]{KechrisBook1995}), we have that 
   \[F\in A\ \Leftrightarrow\ \forall n\in\N,\ \exists H\in \cP,\ H\subseteq F\setminus (-n,n).\]
  So, $A$ is analytic. 

  Let us now notice that $A$ is complete analytic. I.e., we must show that if $Z$ is a Polish space and $B\subseteq Z$ is analytic, then there exists a Borel map $\varphi\colon Z\to \cF(E)$ such that $\varphi^{-1}(A)=B$. By the hypothesis on $E$, we can pick a disjoint sequence $(E_n)_n$ of closed subsets of $E$ such that each $E_n$ is an uncountable subset of $E\setminus (-n,n)$. Notice that, for each $n\in\N$, 
\[\{F\in \cF(E_n)\mid F\text{ is uncountable}\}\]
is $\Sigma^1_1$-complete (see \cite[Theorem 27.5]{KechrisBook1995}). Hence,  for each $n\in\N$, we can pick a Borel map $\varphi_n\colon Z\to \cF(E_n)$ such that 
\[\varphi^{-1}_n(\{F\in \cF(E_n)\mid F\text{ is uncountable}\})=B.\]
Then the map 
\[\varphi\colon z\in Z\mapsto \bigcup_n\varphi_n(z)\in \cF(E)\]
is a Borel map with the desired property.
\end{proof}

\begin{proof}[Proof of Theorem \ref{ThmComplexityCR}]
  For each $F\in \cF(\R)$, there is a canonical embedding of $C(F)$ into $C(\R)$ by simply extending each of the functions in $C(F)$ in an affine way outside $F$. Precisely, if $(a,b)$ is an interval in $\R\setminus F$ maximal with respect to inclusion, then $a,b\in F$ and we extend functions $f$ in $C(F)$ to $(a,b)$ by linearly connecting $(a,f(a))$ to $(b,f(b))$. If $(a,\infty)\subseteq  \R\setminus F$  is maximal with respect to inclusion, then $a\in F$ and we extend functions $f$ in $C(F)$ by letting $f$ be equal to $f(a)$ in $(a,\infty)$. Finally, if $(-\infty,a)\subseteq \R\setminus F$ is maximal with respect to inclusion, then $a\in F$ and we extend functions $f$ in $C(F)$ by letting $f$ be equal to $f(a)$ in $(-\infty,a)$. For simplicity, for each $F\in \cF(\R)$, let us  identify $C(F)$ with this subspace of $C(\R)$. In particular, each $C(F)$ is in $\SF$.
  
  Let us notice that the map 
    \[\varphi\colon F\in \cF(\R)\mapsto C(F)\in \SF\]
    is Borel. For that, let $(\|\cdot\|_n)_n$ be the standard sequence of pseudonorms on $C(\R)$. Let $f\in C(\R)$,   $n\in\N$, $\eps>0$, and set
    \[U(f,n,\eps)=\{g\in C(\R)\mid \|f-g\|_n<\eps\}.\]
    Notice that 
    \[\varphi^{-1}\left(\{X\in \SF\mid X\cap U(f,n,\eps)\neq \emptyset\}\right)=\{F\in \cF(\R)\mid C(F)\cap U(f,n,\eps)\neq \emptyset\}\]
    is Borel. For that, let $(S_n\colon \cF(\R)\to \R)_n$ be a sequence of Borel maps given by Lemma \ref{LemmaKuratowski-Ryll-Nardzewski}, so, $(S_n(F))_n$ is  dense in $ F$ for all $F\in \cF(\R)$. Given  $\bar n=(n_1,\ldots,n_k)\in \N^{<\infty}$ and  $   \bar\alpha=(\alpha_1,\ldots, \alpha_k)\in \Q^{<\infty}$, let $g_{F,\bar n,\bar \alpha}$ be the function in $C(\{S_{n_1}(F),\ldots, S_{n_k}(F)\})$ given by 
    \[g_{F,\bar n,\bar \alpha}(S_{n_i}(F))=\alpha_i\]
    for all $i\in \{1,\ldots, k\}$. Under the identification of  $C(\{S_{n_1}(F),\ldots, S_{n_k}(F)\})$ as a subspace of  $ C(\R)$ described above,  we see $g_{F,\bar n,\bar \alpha}$ as an element in $C(\R)$. As the maps $(S_n)_n$ are Borel, it is  straightforward to check that,  for each $\bar n\in \N^{<\infty}$, each $\bar \alpha\in \Q^{<\infty}$, and each   $t\in \R$, the assignment 
    \[F\in \cF(\R)\mapsto g_{F,\bar n, \bar \alpha}(t)\in \R\]
    is Borel.
    Then, since 
    \begin{align*}
        F\in \varphi^{-1}(\{ X\in \SF \mid  X\cap U(f,n,\eps) & \neq \emptyset\})\\
        \Leftrightarrow\ 
        &\exists \delta\in\Q\cap (0,\eps),\ \exists \bar n\in \N^{<\infty},\  \exists \bar\alpha\in \Q^{<\infty}\\
        &\|f-g_{F,\bar n,\bar \alpha}\|_n\leq\delta,
    \end{align*}
    it follows that $ \varphi^{-1}(\{X\in \SF \mid X\cap U(f,n,\eps)\neq \emptyset\})$ is Borel. As the collection of the sets of the form $U(f,n,\eps)$ form a basis for the topology of $C(\R)$, it follows that $\varphi$ is Borel.

    By Propositions \ref{PropCFnotIsoCR} and \ref{PropCFIsomCR}  we have that \[\varphi(F)\cong C(\R) \ \Leftrightarrow\ F \in  \left\{E\in \cF(\R)\mid \forall n\in\N,\ E\setminus (-n,n)\text{ is uncountable}\right\}.\] So, $\varphi$ is a Borel reduction of $\{E\in \cF(\R)\mid \forall n\in\N,\ E\setminus (-n,n)\text{ is uncountable}\}$ to $\langle C(\R)\rangle$. As $\langle X\rangle$ is analytic for all $X\in \SF$ (Lemma \ref{iso}), Proposition \ref{PropFLargeSimga11Complete} implies that  $\langle C(\R)\rangle $ is $\Sigma^1_1$-complete.
\end{proof}
     
\subsection{Schwartz and Nuclear spaces}\label{someclasses}
The class of Fréchet spaces, being more general than Banach spaces, allow us to consider properties of topological vector spaces which are usually only present in trivial cases for Banach spaces, i.e., when the Banach spaces have finite dimension. In this subsection and in the following one, we compute the complexity of three of the most iconic of such Fréchet properties.   In this subsection specifically, we show that the coding for the class of Schwartz and nuclear spaces are both Borel in $\SF$ (see Theorems \ref{schwartz} and \ref{nuclear}).

Let $X$ be a Fréchet space. A set $K\subseteq X$ is called \emph{totally bounded with respect to a neighborhood $U$ of zero} if for all $\eps>0$ there exists $x_1,...,x_n\in X$ such that $K\subseteq\bigcup_{i=1}^n(x_i+\eps U)$. A Fréchet space $X$ is called a \emph{Schwartz space} if for all neighborhoods $U$ of zero, there exists a neighborhood $V$ of zero which is totally bounded with respect to $U$ (for more on Schwartz spaces see \cite[Chapter 6]{RolewiczSecondEdition1984}).

\begin{thm}\label{schwartz}
Let $\mathrm{Sch}=\{X\in\SF \mid  X\text{ is  Schwartz}\}$. Then $\mathrm{Sch}$ is Borel.
\end{thm}

\begin{proof} First notice that a Fréchet space is a Schwartz space if and only if for all $n\in\N$, all $q\in\R^+$, there exists  $m\in\N$, and $p\in\R^+$, such that $\overline{B}_{m,p}$ is totally bounded with respect to $\overline{B}_{n,q}$. 

Denote by $(S_n)_n$ the sequence of Borel selectors $S_n\colon \SF\to C(\R)$ given by Lemma \ref{LemmaKuratowski-Ryll-Nardzewski}. Precisely, $(S_n)_n$ consists of the restrictions of the Borel selectors $S_n\colon \cF(C(\R))\to C(\R)$ given by Lemma \ref{LemmaKuratowski-Ryll-Nardzewski} to $\SF$.
 Then, given two neighborhoods $V$ and $U$ of zero, if for all $\eps>0$, there exists $n_1,...,n_\ell\in\N$, such that 
\[V\subseteq \bigcup_{i=1}^\ell\big( S_{n_i}(X)+\eps U\big),\]
 then $V$ is totally bounded with respect to $U$.

On the other hand, if for all $n\in\N$, and all $q\in\R^+$, there exists $m\in\N$, and $p\in\R^+$, such that $\overline{B}_{m,p}$ is totally bounded with respect to $\overline{B}_{n,q}$, then, for all $n\in\N$, and all $q\in\R^+$, there exists $m\in\N$, and $p\in\R^+$, such that, for all $\eps>0$, there exists $n_1,...,n_\ell\in\N$, such that 
\[\overline{B}_{m,p}\subseteq \cup_{i=1}^\ell\big( S_{n_i}(X)+\eps \overline{B}_{n,q}\big).\]
This implies that
\begin{align*}
X\in \mathrm{Sch}\Leftrightarrow &\forall q\in\Q^+,\forall n\in\N, \exists p\in\Q^+,\exists m\in\N, \forall \eps\in\Q^+, \exists n_1,...,n_\ell\in\N\\
&\overline{B}_{m,p}(X)\subseteq\bigcup_{i=1}^\ell\big(S_{n_i}(X)+\eps \overline{B}_{n,q}(X)\big).
\end{align*}
 As all the conditions above are Borel, we are done.
\end{proof}

A Fréchet space $X$ is called \emph{nuclear} if all bounded operators from $X$ to any Banach space are \emph{nuclear}. As we will only need a characterization of nuclear spaces that does not depend on the concept of a nuclear operator,  we chose not to define  nuclear operators in these notes. The interested reader can find a detailed exposition of the theory of nuclear spaces  in \cite{RolewiczSecondEdition1984}, Chapters $7$ and $8$. The following theorem gives us the characterization of nuclear spaces that we will need.

Say $(x_j)_j$ is a sequence in a Fréchet space $(X,\|\cdot\|_n)$. Then $(x_j)_j$ is called \emph{absolutely convergent} if, for all $n\in\N$, the series $\sum_j \|x_j\|_n$ converges. The reader can find the following theorem in \cite[Theorem 7.3.2]{RolewiczSecondEdition1984}.

\begin{thm}[Grothendieck] 
Let $X$ be a Fréchet space. Then $X$ is nuclear if and only if every unconditionally convergent series is absolutely convergent.\label{gro}
\end{thm}

\begin{thm}\label{nuclear}
Let $\mathrm{Nuc}=\{X\in \SF\mid X\text{ is nuclear}\}$. Then $\mathrm{Nuc}$ is Borel.
\end{thm}

\begin{lemma}
The set \[\mathrm{UC}=\{(x_n)\in C(\R)^\N\mid (x_n)\text{ is unconditionally convergent}\}\] is coanalytic. 
\end{lemma}

\begin{proof}
This is simply ``counting quantifiers''. Indeed,
\[(x_n)\in \mathrm{UC}\Leftrightarrow \forall (\eps_n)\in \Delta,  \ \Big(\sum_{j=1}^n\eps_j x_j\Big)_n \text{ is Cauchy}.\]
\end{proof}

\begin{lemma}
$\mathrm{Nuc}$ is coanalytic.
\end{lemma}

\begin{proof}
By Grothendieck's characterization of nuclear Fréchet spaces, we have
\begin{align*}
X\in \mathrm{Nuc}\ \Leftrightarrow \ &\forall (x_n)\in \mathrm{UC}\\
& (\forall n,\ x_n\in X)\rightarrow \left(\forall m\ \left(\sum_{j=1}^n\|x_j\|_m\right)_n\text{ is Cauchy}\right).
\end{align*}
 As $\rm UC$ is coanalytic, we have that $\mathrm{Nuc}$ is coanalytic.
\end{proof}

\begin{lemma}
$\mathrm{Nuc}$ is analytic.
\end{lemma}

\begin{proof}
By a result of T. Komura and Y. Komura there exists a universal nuclear space, i.e., there exists $X\in \mathrm{Nuc}$ such that, for all $Y\in \mathrm{Nuc}$, there exists an embedding $Y\hookrightarrow X$ (see \cite{KomuraKomura1965MathAnn} for the original proof  or \cite[Corollary 29.9]{MeiseVogt1997Book}  for further details). Therefore, as a subspace of a nuclear space is nuclear, we have  $\mathrm{Nuc}=\{Y\in \SF\mid Y\hookrightarrow X\}$. As $\{(X,Y)\in \SF^2\mid X\cong Y\}$ is analytic, and as $\{Y\in \SF\mid Y\subseteq X\}$ is Borel, we conclude that $\mathrm{Nuc}$ is analytic.
\end{proof}

\begin{proof}[Proof of Theorem \ref{nuclear}]
As $\mathrm{Nuc}$ is both analytic and coanalytic, Lusin Separation Theorem (see \cite[Theorem 14.7]{KechrisBook1995}) gives us that  $\mathrm{Nuc}$ is Borel. 
\end{proof}

Let 
\[s=\left\{(x_i)_i\in \R^\N\mid \sum_{i\in \N}|x_i|^2i^{2k}<\infty\text{ for all }k\in\N\right\}\]
and endow $s$ with  the Fréchet topology given by the sequence $(\|\cdot\|_k)_k$ of pseudonorms defined by 
\[\|(x_i)_i\|_k=\left( \sum_{i\in \N}|x_i|^2i^{2k}\right)^{1/2}\]
for all $(x_i)_i\in s$ and all $k\in\N$. The Fréchet space $s$ is called the space of \emph{rapidly decreasing sequences}.
\begin{cor}\label{CorsNBorelClass}
    The set 
    \[\{X\in \SF\mid X\text{ embeds into }s^\N\}\]
    is Borel.
\end{cor}
\begin{proof}
    This follows immediately from Theorem \ref{nuclear} since a Fréchet space is nuclear if and only it isomorphically embeds into $s^\N$ (see  \cite{KomuraKomura1965MathAnn} or \cite[Corollary 29.9]{MeiseVogt1997Book}).
\end{proof}

  \subsection{Montel spaces}
  \label{SubsectionMontel}
In this subsection, we continue our study of those Fréchet spaces which are not normable and look at Montel spaces. Recall that a Fréchet space $X$ is called a \emph{Montel space} if all closed bounded subsets of $X$ are compact, in other words, if it satisfies the Heine--Borel property. Banach Montel spaces are just the finite dimensional ones, and therefore the Montel property defines a trivially Borel class in $\SB$. However, and differently from Schwartz and nuclear spaces, we will show that the coding for  Montel spaces in $\SF$ is not Borel. In fact, we show in Theorem \ref{ThmMonCompleteCoanalytic} that it is complete coanalytic.

 We start by showing that Montel spaces are at most coanalytic by simply ``counting quantifiers'':

\begin{prop}\label{montel}
Let $\mathrm{Mon}=\{X\in \SF\mid X\text{ is Montel}\}$. Then $\mathrm{Mon}$ is coanalytic.
\end{prop}

\begin{proof}
Let $B=\{F\in \cF(C(\R))\mid F\text{ is bounded}\}$ and $C=\{F\in \cF(C(\R))\mid F\text{ is compact}\}$. As 
\[F\in B \Leftrightarrow \forall n\in\N,\ \forall p\in\Q^+,\ \exists a\in \Q^+\ \  F\subseteq a \overline{B}_{n,p},\]
 $B$ is Borel. Also, the set of compact subsets of any Polish space is well known to be Borel (see \cite[Exercise 12.11]{KechrisBook1995}). Therefore, we have
\[X\in \mathrm{Mon}\Leftrightarrow \forall F\in B \   (F\subseteq X)\to F\in C,\]
 and $\mathrm{Mon}$ is coanalytic.  
\end{proof}

 In order to show that $\mathrm{Mon}$ is complete coanalytic, we will use \emph{Fréchet sequence spaces}. Loosely speaking, these spaces work as ``Fréchet space versions'' of the classic $c_0$ and $\ell_p$ sequence spaces --- for now, we will only need the ``Fréchet $c_0$ sequence space''. We recall the definition of K\"{o}the matrices.

\begin{definition}\label{DefinitionKothe}
    An $\N$-by-$\N$ matrix   $A=[a_{i,k}]_{(i,k)\in \N\times \N}$ of strictly positive reals is called a \emph{K\"{o}the} matrix if  $a_{i,k}\leq a_{i,k+1}$ for all $i,k\in\N$.\footnote{K\"{o}the matrices are usually assumed to be only matrices of \emph{nonnegative} numbers satisfying the other condition in Definition \ref{DefinitionKothe} and furthermore that  for all $i\in\N$ there is $k\in\N$ with $a_{i,k}\neq 0$. We chose to assume all coordinates of such matrices to be strictly  positive in order  to  avoid (unnecessarily) having to define  the ``inverse of $0$'' and ``multiplication by $\infty$'' (see Theorem \ref{ThmDieudonneGomes}).} Unless some emphasis is needed, we simply write $[a_{i,k}]$ for $[a_{i,k}]_{(i,k)\in \N\times \N}$.
    \end{definition}

We shall now recall the definition of certain Fréchet sequence spaces defined with respect to K\"{o}the matrices. Let $A=[a_{i,k}]$ be such a matrix and let us define a sequence of norms $(\|\cdot\|_k)_k$ on the vector space $c_{00}$ of finitely supported sequences of real numbers. For each $k\in\N$ and $(x_i)_i\in c_{00}$, let 
\[\|(x_i)_i\|_k=\sup_{i\in\N} |x_ia_{i,k}|\]
and let $\lambda(A,k)$ denote the Banach space given by the completion of $(c_{00},\|\cdot\|_k)$.\footnote{Usually, when the numbers $a_{i,k}$ are allowed to be zero, $(\|\cdot\|_k)_k$ are only pseudo-norms, but in our case these are actual norms.} We identify each $\lambda_k(A,k)$ with a vector subspace of   $\R^\N$ --- we emphasize that $\R^\N$ is only considered as a vector space here and not as a Fréchet space --- and let 
\[\lambda_0(A)=\bigcap_{k\in\N}\lambda(A,k).\]
We endow $\lambda_0(A)$ with the Fréchet space structure given by the embedding
\[\lambda_0(A)\to \prod_{k\in \N} \lambda(A,k),\]
where this embedding at each coordinate is simply the inclusion $\lambda_0(A)\hookrightarrow \lambda(A,k)$.\footnote{In the literature $\lambda_0(A)$ is sometimes denoted by $c_0(A)$, e.g., \cite{MeiseVogt1997Book}.}

\begin{remark}
    The spaces $\lambda_0(A)$ should be seen as ``Fréchet $c_0$ sequence spaces''. There are also $\ell_p$ versions of these spaces  and we introduce them in Subsection \ref{SubsectionSDREFL} below.
\end{remark}

 The Dieudonné--Gomes Theorem characterizes when $\lambda_0(A)$ is a Montel space. We refer the reader to  \cite[Theorem 27.9 and Proposition  27.15]{MeiseVogt1997Book} for its proof. 
   
   \begin{theorem}[Dieudonné--Gomes]
 Given a K\"{o}the matrix $A=[a_{i,k}]$, the following assertions about  $\lambda_0(A)$ are equivalent.
 \begin{enumerate}
     \item\label{ThmDieudonneGomesItem1} $\lambda_0(A)$ is a Montel space.
     \item\label{ThmDieudonneGomesItem2} For all infinite $I\subseteq \N$ and all $j\in\N$ there is $k\in \N$ such that $\inf_{i\in I}a_{i,j}a_{i,k}^{-1}=0$.
     \item\label{ThmDieudonneGomesItem3} No infinite dimensional Fréchet subspace of   $\lambda_0(A)$ is normable.
     \item\label{ThmDieudonneGomesItem4} $\lambda_0(A)$ is reflexive. 
 \end{enumerate} \label{ThmDieudonneGomes}
\end{theorem}

Theorem \ref{ThmDieudonneGomes} motivates the next definition.

\begin{definition} \label{DefinitionKotheMontel}
    Let $A=[a_{i,k}]$ be an $\N$-by-$\N$ matrix of positive reals.  We call $A$ a \emph{Montel--K\"{o}the matrix} if $A$ is a K\"{o}the matrix such that      for all infinite $I\subseteq \N$ and all $j\in\N$ there is $k\in \N$ such that $\inf_{i\in I}a_{i,j}a_{i,k}^{-1}=0$.
\end{definition}

In order to show that the coding for the separable Montel spaces is complete coanalytic, we  now show that the subset of Montel--K\"{o}the matrices is complete coanalytic.

 \begin{prop} \label{PropMKComCoanalitic}
 Consider 
 \[ \mathrm{K}=\left\{[a_{i,k}]\in \N^{\N\times \N}\mid [a_{i,k}]  \text{ is a  K\"{o}the matrix}\right\}.\]
     Then $K$ is a Polish space and its  subset 
     \[ \mathrm{MK}=\left\{[a_{i,k}] \in\ \mathrm{K}\mid [a_{i,k}] \text{ is a Montel--K\"{o}the matrix}\right\}\]
     is complete coanalytic.
 \end{prop}

 \begin{proof}
     Since $\N^{\N\times \N} $ is a Polish space, in order to show that $ \mathrm{K}$ is a Polish space, it is enough to notice that $\mathrm{K}$ is a $G_\delta$ (\cite[Theorem 3.11]{KechrisBook1995}). But this is immediate since 
     \[\mathrm{K}=\bigcap_{(i,k)\in \N\times \N}\pi_{i,k}^{-1}\left(\left\{(r,s)\in \N^2\mid r\leq s\right\}\right),\]
     where each $\pi_{i,k}$ is the canonical projection  $ \N^{\N\times\N}\to \N^2$ with respect to the   coordinates $\{(i,k),(i,k+1)\}$. Since each $\pi_{i,k}$ is continuous, the subsets in the intersection above are closed. Hence, they are $G_\delta$ and, consequently, so is $ {\mathrm K}$.

     The fact that $ \mathrm{MK}$ is coanalytic follows easily as well since 
     \[ \mathrm{MK}=\bigcap_{\substack{I\subseteq \N\\ |I|=\infty}}\bigcap_{j\in\N}\bigcup_{k\in\N} \left\{[a_{i,k}]\in \mathrm{K}\mid \inf_{i\in I}a_{i,j}a_{i,k}^{-1}=0\right\}\] 
     and the condition ``$\inf_{i\in I}a_{i,j}a_{i,k}^{-1}=0$'' is Borel as 
     \[\left\{[a_{i,k}]\in \mathrm{K}\mid \inf_{i\in I}a_{i,j}a_{i,k}^{-1}=0\right\}=\bigcap_{\eps\in \Q_+}\bigcup_{i\in I}\left\{[a_{i,k}]\in \mathrm{K}\mid a_{i,j}a_{i,k}^{-1}<\eps\right\}.\]
     
     We shall now show that $ \mathrm{MK}$ is complete coanalytic. Firstly, let
     \[X=\left\{(\lambda_k)_k\in \{n!\mid n\in\N\}^\N\mid  \lambda_k< \lambda_{k+1},\ k\in\N\right\},\]
i.e., $X$ is the set of strictly increasing sequences of factorials. Being a $G_\delta$ subset of the set of all sequences of factorials,  $X$ is a Polish space. Moreover,  it is straightforward to check that $X$ is not  $\sigma$-compact. Therefore, the set
\[S=\left\{(\bar\lambda(i))_i\in X^\N\mid (\bar\lambda(i))_i\text{ does not have a convergent subsequence}\right\}\]
is complete coanalytic (see \cite[Exercise 27.15]{KechrisBook1995}). Let us show that there is a Borel reduction of   $S$ to $ \mathrm{MK}$. Indeed, consider the map $f\colon X^\N\to  \mathrm K$ which takes a sequence $ (\bar\lambda(i))_i$ and forms an $\N$-by-$\N$ matrix  by stacking each of the elements $\bar\lambda(i)$ as a line. Precisely, 
\begin{align*}
    f\colon X^\N&\to  \mathrm{K} \\
    (\bar\lambda(i))_i&\mapsto [\lambda(i)_k]_{(i,k)\in \N\times \N}.
    \end{align*}
   Since each of the sequences $\bar \lambda(i)=(\lambda(i)_k)_k$ is increasing, the image of $f$ is a K\"{o}the matrix as desired.  As $f$ is clearly Borel, we are left to show that $f^{-1}(\mathrm{MK})=S$.

   Suppose $(\bar \lambda(i))_i$ is not in $S$. So, there is an infinite $I\subseteq \N$  such that $(\bar \lambda(i))_{i\in I}$ is convergent in $X$. This means that for each $k\in\N$, the limit 
   \[\lambda_k=\lim_{i\to \infty, i\in I}\lambda(i)_k\]
   exists. Moreover, since the topology in $\{n!\mid n\in\N\}$ is discrete,   for each $k\in\N$ there is $i(k)\in I$ such that 
   \[\lambda(i)_k=\lambda_k\ \text{ for all } i\in I \text{ with }i\geq i(k).\]
   Hence, for each $k\in \N$, we have
   \[\inf_{i\in I}\lambda(i)_1\lambda(i)_{k}^{-1}= \min\left\{\lambda(\ell)_1\lambda(\ell)_{k}^{-1}\mid \ell\in I, \ell\leq \max\{i(1),i(k)\}\right\}>0.\]
   So, $f((\bar \lambda(i))_i)\not\in \mathrm{MK}$.

   Suppose now that $(\bar \lambda(i))_i$ is  in $S$ and fix an infinite $I\subseteq \N$ and $j\in \N$. As $(\bar \lambda(i))_i$  does not have  convergent subsequences, there must be $k\in \N$ such that $(\lambda(i)_k)_{i\in I}$ is unbounded. Indeed, otherwise for each $k\in\N$ there would be $M_k>0$ such that $\lambda(i)_k\in [0,M_k]$ for all $i\in I$. So, for each $i\in I$, $\bar \lambda(i)=(\lambda(i)_k)_k$ would be an element in the compact space $\prod_k[0,M_k]$. In particular, the compactness of this metrizable space   would imply that   $(\bar \lambda(i))_i$ has a convergent subsequence; contradiction. Moreover, since each of the elements of $X$ is an increasing sequence, if $(\lambda(i)_k)_{i\in I}$ is unbounded, so is $(\lambda(i)_{k'})_{i\in I}$ for all $k'\geq k$. Therefore, we can pick $k>j$ such that $(\lambda(i)_k)_{i\in I}$ is unbounded.

We are left to notice that 
\begin{equation}\label{Eq.Inf.}
    \inf_{i\in I}\lambda(i)_j\lambda(i)_k^{-1}=0
    \end{equation} As elements of $X$ are sequences of factorials, for each $i\in I$, we can pick $n(i)$ such that 
\[n(i)!=\lambda(i)_k.\] 
Hence, As $k>j$ and since elements of $X$ are strictly increasing sequences of factorials, we must have that  \[\lambda(i)_j\lambda(i)_k^{-1}\leq \frac{1}{n(i)}\ \text{ for all }i\in I.\]
Finally, as  $(\lambda(i)_k)_{i\in I}$ is unbounded, so is $(n(i))_i$.  Hence, \eqref{Eq.Inf.} holds and we are done.
 \end{proof}

\begin{cor}\label{CorMK*CompleteCoan}
    Consider 
     \[\mathrm{K}^*=\left\{[a_{i,k}]\in \mathrm{K}\mid \forall i\in \N,\forall k\geq i,\ a_{i,k}=a_{i,i}\right\}.\]
     Then $\mathrm K^*$ is a Polish space and its subset 
    \[\mathrm{MK}^*=\mathrm{MK}\cap \mathrm K^*\]
    is complete coanalytic.
\end{cor}

\begin{proof}
    The condition ``$\forall i\in \N,\forall k\geq i,\ a_{i,k}=a_{i,i}$''  defines a $G_\delta$ subset of $\mathrm K$. So, $\mathrm K^*$ is Polish (see \cite[Theorem 3.11]{KechrisBook1995}). Since $\mathrm{MK}$ is coanalytic (Proposition \ref{PropMKComCoanalitic}),  $\mathrm{MK}^*$  is coanalytic. Finally,   $\mathrm{MK}^*$ is complete coanlytic since $\mathrm{MK}$ is complete coanalytic (Proposition \ref{PropMKComCoanalitic}) and the map $f\colon \mathrm K\to \mathrm K^*$  given by letting
    \[f(A)_{i,k}= \left\{\begin{array}{ll}
a_{i,k},& \text{ if } i\geq k  \\
      a_{i,i}   ,& \text{ if } i<k, 
    \end{array}\right.\]
    for all $A=[a_{i,k}]\in \mathrm K$, is a Borel reduction of $\mathrm{MK}$ to $\mathrm{MK}^*$. Indeed, it is completely straightforward to check that  $f$ is Borel and that $f^{-1}(\mathrm{MK}^*)=\mathrm{MK}$;  we leave this trouble to the reader.
\end{proof}
The following is the main result of this subsection.
\begin{theorem}\label{ThmMonCompleteCoanalytic}
    $\mathrm{Mon}$ is a complete coanalytic subset of $\SF$. In particular, $\mathrm{Mon}$ is not analytic.
\end{theorem}

Before proving Theorem \ref{ThmMonCompleteCoanalytic}, we isolate the main technical lemma needed for its proof. 

\begin{lemma}\label{LemmaFunctionReducingMontel}
    There is a Borel map $\varphi \colon \mathrm K^*\to \SF$ such that $\varphi(A)$ is isomorphic to $\lambda_0(A)$ for all $A\in \mathrm K^*$.
\end{lemma}

\begin{proof}

For each $i\in\N$, let 
\[\N(i)=\left\{\bar n=(i,(n_k)_k)\in \{i\}\times \N^\N\mid n_k=n_i,\ \forall k\geq i\right\}.\]
 We then let
\[T=\bigsqcup_{i\in\N}\N(i),\]
so, $T$ is countable.  Let
\[c_{00}(T)=\left\{(x_{\bar n})_{\bar n\in T}\in \R^T\mid |\{\bar n\in T\mid x_{\bar n}\neq 0\}|<\infty\right\}.\]  We define a sequence of norms $(\|\cdot\|_k)_k$ on $c_{00}(T)$ by letting 
\[\|(x_{\bar n})_{\bar n\in T}\|_k=\sup \left\{ |x_{\bar n}n_k|\mid \bar n=(i,(n_\ell)_\ell)\in T\right\}\]
for each $k\in\N$ and each $(x_{\bar n})_{\bar n\in T}\in c_{00}(T)$. Given $E\subseteq T$, we let 
\[c_{00}(E)=\left\{(x_{\bar n})_{\bar n\in T}\in c_{00}(T)\mid  x_{\bar n}= 0, \forall \bar n\in T\setminus E\right\}\]
and, for each $k\in\N$, we define  $\Lambda(E,k)$  as the completion of $(c_{00}(E),\|\cdot\|_k)$. We see each  $\Lambda(E,k)$ as a Banach subspace of $\Lambda(T,k)$; precisely, $\Lambda(E,k)$ is the $\|\cdot\|_k$-norm closure of $c_{00}(E)$ in $\Lambda(T,k)$. Identifying each $\Lambda(T,k)$ with a vector subspace of the vector space $\R^T$, all of $\Lambda(E,k)$ are seen as vector subspaces of $\R^T$.  We can then define 
\[\Lambda(E)=\bigcap_k\Lambda(E,k)\]
and endow $\Lambda(E)$ with   the Fréchet space structure given by the   embedding 
\[\Lambda(E)\to\prod_{k\in \N}\Lambda(E,k),\]
 where this embedding in each coordinate if simply the inclusion $\Lambda(E)\hookrightarrow \Lambda(E,k)$.
By the indentifications above, each $\Lambda(E)$ is a Fréchet subspace of $\Lambda(T)$.

For each $A=[a_{i,k}]\in\mathrm K^*$ and $i\in \N$, let $\bar n(i)=(i,(n(i)_k)_k)$ denote the element of $\N(i)$ given by the $i$th line of $A$, i.e.,  \[  n(i)_k=a_{i,k}\ \text{ for all }\ i,k\in\N.\] Then set
\[E(A)=\bigsqcup_{i\in\N}\{\bar n(i)\}.\]
So, $E(A)\subseteq T$.

By simply unfolding definitions, we have that  $\Lambda(E(A))$ is isomorphic to $\lambda_0(A)$ for all $A\in \mathrm{K}^*$. Therefore, we are left to notice that the assignment $A\mapsto \Lambda(E(A))$ can be done in a Borel manner.   As $T$ is countable, $\Lambda(T)$ is separable. Therefore, we can fix a Fréchet space $X\in \SF$ isomorphic to $\Lambda(T)$. Let $F\colon \Lambda(T)\to X$ be such isomorphism. Since for each $E\subseteq T$,   $\Lambda(E)$ is a Fréchet subspace of $\Lambda(T)$, we can   define a map $\varphi\colon \mathrm{K}^*\to \SF$ by letting 
\[\varphi(A)=F[\Lambda(E(A))]\]
for all $A\in \mathrm K^*$. This construction immediately gives that   $\varphi(A)$ is isomorphic to $\Lambda(E(A))$ for all $A\in \mathrm{K}^*$, so, we only need to show that $\varphi$ is Borel. 
 
 Let $U\subseteq C(\R)$ be an open set and let us show that 
$\{A\in \mathrm{K}^*\mid \varphi(A)\cap U\neq \emptyset\}$
is Borel. Let 
\[c_{00}(T,\Q)=c_{00}(T)\cap \Q^T.\]
Notice that $\varphi(A)\cap U\neq \emptyset$ if and only if there is $(x_{\bar n})_{\bar n\in T}$ in $c_{00}(T,\Q)$  whose image under $F$ is in  $ U$ and such that  $x_{\bar n}= 0$ for all $\bar n \in  T\setminus E(A)$.   Therefore, letting   
\[I=\left\{(x_{\bar n})_{\bar n\in T}\in c_{00}(T,\Q)\mid F\left((x_{\bar n})_{\bar n\in T}\right)\in U\right\},\]
we have that 
\[\{A\in \mathrm{K}^*\mid \varphi(A)\cap U\neq \emptyset\}=\bigcup_{(x_{\bar n})_{\bar n\in T}\in I}\bigcap_{\substack{\bar n\in T, \\
x_{\bar n}\neq 0}} \left\{A\in \mathrm{K}^*\mid \bar n\in E(A)\right\}. \]
As all sets in the right-hand side above are Borel and the indices in the union and intersection are countable, this shows that  the set of the left-had side above is also Borel as desired.
\end{proof}

\begin{proof}
    [Proof of Theorem \ref{ThmMonCompleteCoanalytic}]
    By Proposition \ref{montel}, $\mathrm{Mon}$ is coanalytic. So, we are left to notice that   any coanalytic set Borel reduces to $\mathrm{Mon}$. By Lemma \ref{LemmaFunctionReducingMontel}, there is a  Borel map $\varphi \colon \mathrm K^*\to \SF$ such that $\varphi(A)$ is isomorphic to $\lambda_0(A)$ for all $A\in \mathrm K^*$.  Hence, since the equivalence between \eqref{ThmDieudonneGomesItem1} and \eqref{ThmDieudonneGomesItem2} of Theorem \ref{ThmDieudonneGomes} says that $\lambda_0(A)$ is Montel if and only if $A\in \rm MK$, this function is a Borel reduction of $\rm MK^*$ to $\mathrm{Mon}$. As $\rm MK^*$  is complete coanalytic (Corollary \ref{CorMK*CompleteCoan}), the result follows.
\end{proof}

\begin{theorem}\label{ThmContainsInfDimComAna}
    The subset 
    \[\{X\in \SF\mid X\text{ contains an infinite dimensional Banach space}\}\]
    is complete analytic.
\end{theorem}

\begin{proof}
If $A$ denotes the set in the statement of the theorem and $\SB_\infty$ is the subset of $\SB$ of infinite dimensional spaces, then 
\[X\in A\ \Leftrightarrow \ \exists Y\in \SB_\infty\ \text{  such that }  \ Y\hookrightarrow X.\]
As $\SB_\infty $ is Borel, it follows that $A$ is analytic. The fact that it is complete analytic then follows as the proof of Theorem \ref{ThmMonCompleteCoanalytic} with the only difference being that instead of using the equivalence of \eqref{ThmDieudonneGomesItem1} and \eqref{ThmDieudonneGomesItem2} of Theorem \ref{ThmDieudonneGomes}, we replace this with the equivalence between \eqref{ThmDieudonneGomesItem1} and \eqref{ThmDieudonneGomesItem3} of Theorem \ref{ThmDieudonneGomes}.
\end{proof}

\begin{cor}\label{CorNoMontelUniversal}
    There is no  Montel space which is isomorphically universal for all separable Montel spaces.
\end{cor}

\begin{proof}
    Suppose towards a contradiction that there is $X\in \mathrm{Mon}$ such that every $Y\in \mathrm{Mon}$ isomorphically embeds into $X$. Then, since closed subspaces of Montel spaces are Montel, this implies that 
    \[\mathrm{Mon}=\langle \SF(X)\rangle=\{Y\in \SF\mid \exists Z\in \SF(X),\ Z\cong Y\}.\]
    Since $\SF(X)$ is Borel, $\langle \SF(X)\rangle$ is analytic by Lemma \ref{iso}. This contradicts Theorem \ref{ThmMonCompleteCoanalytic}.
\end{proof}

\begin{cor}\label{CorNoTrulyFrechetUniversal}
 If a separable Fréchet space contains isomorphic copies of all separable Montel spaces, then it must also contain an isomorphic copy of an infinite dimensional Banach space.
\end{cor}

\begin{proof}
Suppose $X\in \SF$ contains isomorphic copies of all separable Montel spaces and let $\varphi\colon \rm K^*\to \SF$ be the Borel map given by Lemma \ref{LemmaFunctionReducingMontel}. Then \[\varphi(\rm MK^*)\subseteq \langle \SB(X)\rangle \subseteq \SF.\]
As $\langle \SB(X)\rangle $ is analytic (Lemma \ref{iso}), and $\rm MK^*$ is not (Corollary \ref{CorMK*CompleteCoan}), the preimage of $\langle \SB(X)\rangle $ under $\varphi$ must intersect $\rm K^*\setminus \rm MK^*$. Picking $A$ in this intersection, we have that $\varphi(A)$ is isomorphic to a subspace of $X$ and $A$ is not in $\rm MK^*$. By Theorem \ref{ThmDieudonneGomes}, $\varphi(A)$ must contain an isomorphic copy of some infinite dimensional Banach space and hence so does $X$.
\end{proof}

\subsection{Revisiting $\mathrm{REFL}_{\neg \rm B}$ and $\SD_{\neg \rm B}$}\label{SubsectionSDREFL} 
  We have seen in Corollary \ref{CorollarySDandReflNotBorel} that  neither $\mathrm{REFL}_{\neg\rm B}$ nor $\SD_{\neg \rm B}$ are Borel. Now, with the methods of Subsection \ref{SubsectionMontel}, we can improve this result and show that both these sets are in fact \emph{$\Pi^1_1$-hard}, i.e., every coanalytic set Borel reduces to it. For  $\mathrm{REFL}_{\neg\rm B}$, the proof is essentially already in Subsection \ref{SubsectionMontel}. The result for $\SD_{\neg \rm B}$  will require a mild adaptation.  Indeed, for  $\mathrm{REFL}_{\neg\rm B}$, the Borel map constructed in Lemma \ref{LemmaFunctionReducingMontel} also gives us this result as consequence. Precisely:

\begin{theorem}\label{ThmREFLnotBanachPI11Hard}
The subset   $\mathrm{REFL}_{\neg \rm B}$ is $\Pi^1_1$-hard.
\end{theorem}
\begin{proof}
      Once again,  this proof is  the same as the one of Theorem \ref{ThmMonCompleteCoanalytic}  with the only difference now being that instead of using the equivalence   \eqref{ThmDieudonneGomesItem1}$\Leftrightarrow$\eqref{ThmDieudonneGomesItem2} of Theorem \ref{ThmDieudonneGomes}, we use its  equivalence \eqref{ThmDieudonneGomesItem1}$\Leftrightarrow$\eqref{ThmDieudonneGomesItem4}.
\end{proof}

We shall now outline the modifications needed to show that   $\SD_{\neg \rm B}$ is $\Pi^1_1$-hard. For that, let us first explain why some modification is needed. Let  $\varphi \colon \rm K^*\to \SF$ be the map defined in  Lemma \ref{LemmaFunctionReducingMontel}. So, $\varphi$ is a Borel reduction of $\rm MK^*$ to $\rm Mon$ such that $\varphi(A)$ is isomorphic to $\lambda_0(A)$ for all $A\in \rm K^*$.  By Theorem \ref{ThmDieudonneGomes},  $\lambda_0(A)$ contains an infinite dimensional Banach space if $A\not\in \rm MK^*$. The reason for that is actually quite simple: if $A=[a_{i,j}]\in \rm K$ is so  that there is $I\subseteq \N$ and $j\in\N$ such that for all $k\geq j$ we have $\inf_{i\in I}a_{i,j}a^{-1}_{i,k}>0$, then 
the topology of  
\[Y=\left\{(x_i)_i\in \lambda_0(A)\mid\forall i\in \N\setminus I,\  x_i=0\right\}\]
is completely determined by the norm $\|\cdot\|_j$. Therefore, the map 
\[(x_i)_i\in Y\mapsto (x_ia_{i,j})_i\in c_0\]
is an isomorphism and we conclude not only that $\lambda_0(A)$  contains \emph{some} infinite dimensional Banach space, but actually that $\lambda_0(A)$ contains an isomorphic copy of $c_0$. As a subproduct of this discussion, we isolate a strengthening of Corollary \ref{CorNoTrulyFrechetUniversal}.

\begin{cor}\label{CorNoTrulyFrechetUniversal.V2}
 If a separable Fréchet space contains isomorphic copies of all separable Montel spaces, then it must also contain an isomorphic copy of $c_0$.\qed
\end{cor}

As some readers may have already noticed, the issue to proceed exactly as in Theorem
\ref{ThmREFLnotBanachPI11Hard} to show that $\SD_{\neg \rm B}$ is $\Pi^1_1$-hard is simply  that $c_0$ has a separable dual. However, this is easy to go around. The Fréchet space $\lambda_0(A)$ is a Fréchet sequence space constructed to be a not normable version of  $c_0$. There is a straightforward modification of its construction which gives us ``Fréchet $\ell_p$-sequence spaces''. For $p=1$, this construction will give us what we need. 

Let us quickly define the $\lambda_p(A)$ spaces. Let $p\in [1,\infty)$  and  $A=[a_{i,k}]\in \rm K$. Then $\lambda_p(A)$ is the vector subspace of $\R^\N$ consisting of all $(x_i)_i\in \R^\N$ such that \[\|(x_i)_i\|_k=\left(\sum_{i\in\N} |x_ia_{i,k}|^p\right)^{1/p}<\infty\ \text{ for all }\ k\in\N.\]
We endow $\lambda_p(A)$ with the Fréchet structure given by all norms $(\|\cdot\|_k)_k$.  

Proceeding exactly as in Subsection \ref{SubsectionMontel} with the space $\lambda_p(A)$ instead of $\lambda_0(A)$, we obtain the following version of Lemma \ref{LemmaFunctionReducingMontel}:

\begin{lemma}\label{LemmaFunctionReducingMontel.V2}
Let $p\in [1,\infty)$. There is a Borel map $\varphi\colon \rm K^*\to \SF$ such that $\varphi(A)$ is isomorphic to $\lambda_p(A)$ for all $A\in \rm K^*$.    
\end{lemma}
By the discussion above and since  Theorem \ref{ThmDieudonneGomes} also holds verbatim for $p\in[1,\infty)$ (see \cite[Theorem 27.9 and Proposition and 27.15]{MeiseVogt1997Book}), the map $\varphi$ of Lemma \ref{LemmaFunctionReducingMontel.V2} satisfy:
\begin{itemize}
\item $\lambda_p(A)$ is Montel for all $A\in \rm MK^*$ and 
    \item $\lambda_p(A)$ contains an isomorphic copy of $\ell_p$ for all $A\in \rm K^*\setminus \rm MK^*$. 
\end{itemize} 
Therefore, 
 Corollary \ref{CorNoTrulyFrechetUniversal.V2} is strengthened as follows.

\begin{cor}\label{CorNoTrulyFrechetUniversal.V3}
 If a separable Fréchet space contains isomorphic copies of all separable Montel spaces, then it must also contain   isomorphics copy of $c_0$ and of $\ell_p$ for all $p\in[1,\infty)$.\qed
\end{cor}

Moreover, proceeding as in Theorem \ref{ThmREFLnotBanachPI11Hard}, Lemma \ref{LemmaFunctionReducingMontel.V2} for $p=1$ gives us the following:
 
\begin{theorem}\label{ThmSDnotBanachPI11Hard}
The subset   $\mathrm{SD}_{\neg \rm B}$ is $\Pi^1_1$-hard.\qed
\end{theorem}

\begin{remark}
    Some readers may be wondering why we have not presented Subsection \ref{SubsectionMontel} with $\lambda_p(A)$ right away instead of  only working with $\lambda_0(A)$ therein. This is a fair question, the answer of which may or may not be satisfactory: we believe the construction of the spaces $\lambda_0(A)$ and of the Borel map $\varphi\colon \rm K^*\to \SF$ was already technical enough without having to deal both with $c_0$ and $\ell_p$ at the same time. We prefered to focus on the $c_0$-case first which has the cleanest notation and then explain its trivial modification at this later stage. 
\end{remark} 

\subsection{Fréchet spaces with a basis}\label{subsectionbasis}
One of the most important problems in descriptive set theory of separable Banach spaces is the \emph{basis problem}: is the analytic subset of Banach spaces with Schauder basis, 
\[\SB_\mathrm{b}=\{X\in \SB\mid X\text{ has a Schauder basis}\},\]
not Borel? This short subsection serves as further advertisement for this problem by introducing it to the Fréchet spaces scenario.

 We start recalling a lemma (see \cite[Corollary 2.6.5]{RolewiczSecondEdition1984}).


\begin{lemma}
Let $X$ be a Fréchet space. A sequence $(x_j)_j$ in $X$ is basic if and only if it is linearly independent, and the sequence of operators $(P_j)_{j}$ defined by $P_j(\sum_ia_ix_i)=\sum_{i=1}^ja_ix_i$ (whenever $\sum_ia_ix_i$ converges, where $(a_i)$ is a sequence of real numbers) is equicontinuous.
\end{lemma} 

\begin{thm}\label{basis}
The subset \[\SF_b=\{X\in\SF\mid X\text{ has a basis}\}\] is analytic in $\SF$.
\end{thm}

\begin{proof}
We have that
\[X\in\SF_\mathrm{b}\Leftrightarrow \exists (x_n)\in C(\R)^\N\ \big(\forall n\ x_n\in X\big)\wedge \big((x_n)\text{ is a basis for }X\big).\]
 A sequence $(x_j)_j$ in $X$ is a basis for $X$ if and only if it is linearly independent (L.I.), the sequence of projection $(P_j)_j$ is equicontinuous, and $\text{span}\{x_j\mid j\in\N\}$ is dense in $X$. Therefore, the only thing left to show is that each of those three conditions can be written in a Borel manner.

By the compactness of the finite dimensional ball, we have
\begin{align*}
(x_j)_j\in X^\N\text{ is L.I.}\ \Leftrightarrow \ &\forall n\in\N, \forall \delta\in\Q^+\cap(0,1], \exists \eps\in \Q^+\\
& \forall (a_1,...,a_n)\in\Q^{<\N} \ (\text{s.t.\ }\ \exists i\ a_i\geq\delta),\\
&\left\|\sum_{j=1}^na_jx_j\right\|_F\geq \eps.
\end{align*} 

By the definition of equicontinuity, we have that, considering the sequence of projections $(P_j)_j$ given by a linearly independent sequence $(x_j)_j$, 
\begin{align*}
(P_j)_j\text{ is equicontinuous}\ \Leftrightarrow\ &\forall m\in\N,\forall q\in\Q^+,\exists n\in\N, \exists p\in\Q^+,\\
&\forall j\in\N, \forall (a_1,...,a_\ell)\in\Q^{<\N}\\
&\sum_{k=1}^\ell a_kx_k\in B_{n,p}\Rightarrow P_j\left(\sum_{k=1}^\ell a_kx_k\right)\in B_{m,q}.
\end{align*}

The relation ``$\text{span}\{x_j\mid j\in\N\}$ is dense in $X"$ is well known to be Borel (see \cite[page 10, \textbf{(P5)}]{DodosBook2010}). We are done.
\end{proof}

 \section{Open problems}\label{SecOpenProb}
As this paper introduces the descriptive theory for the class of Fréchet spaces, this is an open field full of interesting problems to be solved. In here, we list very few of them which have called our attention during the process of writing this article.

The problem of whether the coding for the Banach spaces with a Schauder basis is Borel is arguably one of the most important problems in the field. We then ask the following: 

 \begin{problem}
Is $\SF_\mathrm{b}$ Borel? If $\SF_\mathrm{b}$ is non Borel, is it complete analytic?  
\end{problem}

A locally convex topological vector space $X$ is called \emph{barreled} if every subset of $X$ which is closed, convex, absorbing, and balanced is a neighborhood of $0$. We then say that a Fréchet spaces $X$ is \emph{distinguished} if $X^*$, its strong dual,  is barreled. 

\begin{problem}
   Consider 
    \[\mathrm{D}=\{X\in \SF\mid X\ \text{ is distinguished}\}.\]
What is the precise complexity of $\rm D$?\end{problem}

 There is a sensitive reason to expect that $\mathrm D$ is at least  $\Pi^1_2$-hard; hence, a very complex class. This is motivated by  \cite[Proposition 27.17]{MeiseVogt1997Book}. This result characterizes the K\"{o}the matrices $A$ for which $\lambda_1(A)$ is distinguished and this quantification is of complexity $\Pi^1_2$ in nature, i.e., a universal quantifier over a Polish space followed by an existential one.

\section{Appendix}\label{appendix}

The coding of separable Fréchet spaces $\SF$ is by no means unique. Therefore, a natural question is whether different codings can give  different complexities for a given class of separable Fréchet spaces $\cP$. For completeness of these notes, in this appendix we recall the argument of B.  Bossard which shows that,  for reasonable enough codings, this does not occur. 

First, we define a \emph{coding of separable Fréchet spaces up to isomorphism} as a pair $(c,E)$, where $E$ is a set and  $c$ is a function from $E$ onto the quotient $\SF/{\cong}$. The coding we were working with in the previous sections is the coding $E=\SF$ and $c(X)=\langle X\rangle$, where $\langle X\rangle$ denotes the equivalence class of Fréchet spaces isomorphic to $X$. 

\begin{prop}\emph{(}\cite[Proposition 2.8]{Bossard2002}\emph{).} 
Let $E$ and $F$ be standard Borel spaces, and $G$ be a set. Let $c_1\colon E\to G$ and $c_2\colon F\to G$ be surjective functions. If the set $\{(x,y)\in E\times F\mid c_1(x)=c_2(y)\}$ is analytic we have

\begin{enumerate}
\item \label{absconsistent.1} If $A\subseteq E$ is $\Sigma_n^1$, for some $n\in\N$, so is $c_2^{-1}(c_1(A))$.
\item \label{absconsistent.2}  Let $C\subseteq G$, and $n\in\N$. Then $c_1^{-1}(C)$ is $\Sigma^1_n$ (resp. $\Pi^1_n$) in $E$ if and only if $c_2^{-1}(C)$ is $\Sigma^1_n$ (resp. $\Pi^1_n$) in $F$.  
\end{enumerate}
In particular, Lusin's separation theorem implies that $c_1^{-1}(C)$ is Borel in $E$ if and only if $c_2^{-1}(C)$ is Borel in $F$.\label{absconsistent}
\end{prop}

Let us apply the proposition above to another natural coding of the separable Fréchet spaces to see that the complexity of classes of separable Fréchet spaces coincides in both codings.

Let $\tilde{c}\colon C(\R)^\N\to\SF/{\cong}$ be such that $c((f_n)_n)=\overline{\text{span}}\{f_n\mid n\in\N\}$. The pair $(\tilde{c},C(\R)^\N)$ is another natural coding for $\SF$. Let $(c,\SF)$ be our regular coding and let us observe that the functions $c\colon \SF\to\SF/{\cong}$ and $\tilde{c}\colon C(\R)^\N\to\SF/{\cong}$ fit the assumptions of Proposition \ref{absconsistent}. Indeed, let $\Phi$ be the function defined in Lemma \ref{phi}. We then have 
\begin{align*} 
\{(X,(f_n)_n)\in\SF\times C(\R)^\N\mid & c(X)
=\tilde{c}((f_n)_n)\}\\
&=(\rm{Id},\Phi)^{-1}(\{(X,Y)\in\SF^2\mid X\cong Y\}),
\end{align*}
 where $\rm{Id}$ denotes the identity on $\SF$. By Lemma \ref{iso} and Lemma \ref{phi}, we conclude that $\{(X,(f_n)_n)\in\SF\times C(\R)^\N\mid c(X)=\tilde{c}((f_n)_n)\}$ is analytic.  Therefore, Proposition \ref{absconsistent} implies that both codings $(\tilde{c},C(\R)^\N)$ and $(c,\SF)$ give us the same complexity for any given class of separable Fréchet spaces.\\

\noindent \textbf{Acknowledgments:} The first named author would like to thank Joe Diestel for his suggestion of looking at separable Fréchet spaces in this context back in 2013.  The authors are thankful to José Bonet, Gilles Godefroy, Willian B. Johnson, Gillies Lancien, Quentin Menet,  and Dietmar Vogt for conversations about the state of art of the Fréchet space theory.

\bibliographystyle{alpha}
\bibliography{bibliography}
 \end{document}